\documentclass{amsart}

\usepackage[hmargin=3.5cm,top=3cm,bottom=3cm,includehead]{geometry}

\usepackage{mathabx} % for '\widecheck' macro

\usepackage{amsmath, amssymb, amsthm, mathrsfs}
\usepackage{graphicx,enumitem} 
\usepackage{cancel}
 
\usepackage{xcolor}
\usepackage[pdfborder={0 0 0}]{hyperref}

\newcommand\R{{\mathbb R}}

\def\AA{{\mathcal A}}
\def\BB{{\mathcal B}}

\def\DD{{\mathcal D}}

\def\FF{{\mathcal F}}

\def\HH{{\mathcal H}}

\def\LL{{\mathcal L}}

\def\NN{{\mathcal N}}
\def\OO{{\mathcal O}}

\def\UU{{\mathcal U}}

\def\BBB{{\mathscr B}}

\def\EEE{{\mathscr E}}

\def\HHH{{\mathscr H}}

\def\LLL{{\mathscr L}}
\def\MMM{{\mathscr M}}

\def\OOO{{\mathscr O}}

\def\RRR{{\mathscr R}}

\def\OOO{{\frak O}}
\def\cc{{\frak c}}

\def\eps{{\varepsilon}}
\def\wto{{\,\rightharpoonup\,}}
\def\Lloc{L_{\rm loc}}

\def\fraka{{{\mathfrak {a}}}}
\def\frakb{{\frak b}}
\def\frakP{{\frak P}}

\newcommand{\Nt}{|\hskip-0.04cm|\hskip-0.04cm|}

\newtheorem{theo}{Theorem}[section]
\newtheorem{prop}[theo]{Proposition}
\newtheorem{lem}[theo]{Lemma}
\newtheorem{cor}[theo]{Corollary}
\newtheorem*{thm*}{Theorem}
\theoremstyle{remark}

\newtheorem*{ex*}{Example}
\theoremstyle{definition}

\numberwithin{equation}{section}
\newcommand{\be}{\begin{equation}}
\newcommand{\ee}{\end{equation}}
\newcommand{\ba}{\begin{aligned}}
\newcommand{\ea}{\end{aligned}}
\newcommand{\beqn}{\begin{equation}}
\newcommand{\eeqn}{\end{equation}}
\newcommand{\bear}{\begin{eqnarray}}
\newcommand{\eear}{\end{eqnarray}}
\newcommand{\bean}{\begin{eqnarray*}}
\newcommand{\eean}{\end{eqnarray*}}

\newcommand{\Black}{\color{black}}

\title%{Nonlinear asymptotic stability for }
{On the Voltage-Conductance kinetic equation %\\ in a small conductivity regime}
 }
\date{\today}

\author[C. Fonte Sanchez]{Claudia Fonte Sanchez}

 \author[S. Mischler]{St{\'e}phane Mischler}%St\'ephane Mischler

\address[S.~Mischler]{Centre de Recherche en Math\'ematiques de
  la D\'ecision (CEREMADE, CNRS UMR 7534),
  Universit\'es PSL \& Paris-Dauphine, Place de Lattre de
  Tassigny, 75775 Paris 16, France \& Institut Universitaire de France (IUF)}
\email{mischler@ceremade.dauphine.fr}

\address[C. Fonte Sanchez]{Universit\'e Grenoble Alpes, Inria, 38000 Grenoble, France}

\email{claudia.fonte-sanchez@inria.fr}



\begin{document}

\begin{abstract}
We consider the nonlinear Voltage-Conductance kinetic equation arising in neuroscience. 
We establish the existence of solutions in a weighted $L^\infty$ framework in a weak interaction regime. 
%for any nonnegative conductivity parameter. 
%For a zero conductivity parameter, we prove the 
We also prove the linear asymptotic exponential  stability of the steady state making constructive 
a recent estimate of \cite{dou2023voltageconductance}. %The two 
Both results are based in a fundamental way on some ultracontractivity property of the
flow associated to the  linear (possibly time dependent) Voltage-Conductance kinetic equation. 
\end{abstract}

\maketitle

\tableofcontents

\section{Introduction}

The integrate-and-fire (IF) neuron model has proven to be a valuable tool in computational neuroscience for simulating the behavior of individual cortical neurons \cite{mclaughlin2000neuronal, rangan2005modeling,brette2005adaptive}. However, scaling this approach to model large cortical regions presents significant challenges.  Both the computational demands of simulating vast networks and the qualitative analysis of their complex dynamics become increasingly difficult.
 In this context, the Voltage-Conductance kinetic (VCk) equation emerges as a population-level representation of the neuronal activity \cite{CaiTSMcL,MR2204080,MR2293150,MR2495462}, analogous to the approach of classical kinetic theory for charged particles \cite{MR0875086,MR1200643,MR1634851,MR2721875,MR3778533}, and as an alternative to the  kinetic Fitzhugh-Nagumo model  \cite{FitzHugh,Nagumo,MR3465438,MR4061402}.  

Consider a network of IF neurons, where each neuron (indexed by $j$ out of $N$ total) is described by two state variables: voltage, $v^j$, and conductance, $y^j$.  The dynamics of the $j$th neuron during periods between spikes are governed by the following equations:
\bear \label{eq_sing1}
\frac{dv^j}{dt}&=&-(v^j-v_R)-y^j(v^j-v_E),\\ \label{eq_sing2}
\frac{dy^j}{dt}&=&-y^j+y_\ast+\frac{\cc}{N}\sum_{l,k\neq j}\delta(t-t_k^\ell). 
\eear
Here, $v_R$ represents the resting voltage (set to 0 for convenience), $v_E$ is the excitatory reversal potential, and $y_\ast>0$ corresponds to the constant external stimulus originating in the lateral geniculate nucleus (LGN) from the visual sensors (e.g., \cite{CaiTSMcL}).  The term $\delta(t-t_k^\ell)$ in (\ref{eq_sing2}) represents a Dirac delta function that models the synaptic input received by neuron $j$ from the $\ell$th spike of neuron $k$. The parameter $\cc\geq0$ scales the strength of these synaptic interactions.

In the absence of any external or network input, the voltage of the neuron naturally decays towards its resting state ($v_R$) due to the first term in the first equation. The neuron can, however, receive excitatory input from two sources, the external stimulus ($y_\ast$) and the input that arises from the activity of other neurons in the network, captured by the summation term in the second equation. These combined inputs influence the voltage of the neuron indirectly by affecting its conductance ($y^j$). The conductance acts like a gain control mechanism, modulating how effectively the neuron translates these inputs into changes in its voltage.
In complement with this dynamic,  when the neuron's voltage reaches a threshold, $v_F$, it is assumed to spike.  At this point (say, time t), the voltage is reset to the resting value, $v^j(t^+)=v_R$.

The VCk model is derived from this system by applying a closure in the mean-field limit of a large number of weakly connected neurons \cite{CaiTSMcL,MR2204080,MR2293150}. It describe the evolution of the neuronal network though the probability $F_t(v,y)$ to find neurons at time $t\ge 0$ with a membrane potential $v \in (0,v_F)$ and a conductance $y \in (0,\infty)$. We refer to \cite{MR3177631,MR4548856,salort:hal-03845918,dou2023voltageconductance} for previous mathematical analysis papers, discussions and further references.

\subsection{The VCk equation} 
The core mathematical formulation of the VCk model is a partial differential equation (PDE) known as the Voltage-Conductance kinetic equation. This equation governs the time evolution of the probability density function  $F_t(v,y)$ and reads 
\beqn\label{eq:VCk}
\partial_t F  + \partial_v(JF) + \partial_y(K_F F) - a_F \partial_{yy}^2 F = 0
\quad\text{in}\quad (0,\infty) \times \OO
\eeqn
with the shorthand $(v,y) \in \OO := (0,v_F) \times (0,\infty) $ and where the the coefficients are given by 
 \bean
 J &:=& y (v_E-v) - y_L v, 
 \\
 K_F &:=& y_* + \cc\, \NN_F - y, 
 \\
a_F &:=& a_* + \cc^2 \NN_F , 
\eean
 with $0 < v_F < v_E$,  $y_L,a_*, y_* > 0$, $\cc \ge 0$. The definition of $J$ aligns with equation \eqref{eq_sing1}, while $K_F$ and $a_F$ are derived from equation \eqref{eq_sing2} and the closure process. The interaction term in \eqref{eq_sing2} becomes a global firing rate $\NN_F$ of the limit system. It influences the evolution of $F$ through two mechanisms: a current and an internal noise. For $F_t : \bar\OO \to \R$, the global firing rate is defined as the current generated by all the neurons spiking at time $t$,  that is 
 \beqn
 \label{eq:defNNF}
 \NN_F(t) := \int_0^\infty J(v_F,y) F_t(v_F,y) dy. 
 \eeqn
It is worth noticing that the equation's nonlinearity, driven by the term $\NN_F$, is controlled by the connectivity parameter $\cc$. 
In fact, the equation becomes linear when there is no connection between neurons, which corresponds to $\cc = 0$. Additionally, similar to classical kinetic models, the noise term in the VCk equation solely affects the conductance variable.

\smallskip
 We complement the VCk equation \eqref{eq:VCk} with  an initial condition
\beqn\label{eq:VCkt=0}
F(0,\cdot) = F_0\quad\text{in}\quad \OO,
 \eeqn
specifying the probability distribution at time zero. Additionally, boundary conditions are imposed to capture specific behaviors:
  \bear \label{eq:VCktBd2}
&& F = 0  \ \hbox{ on } \ \Gamma_1 := (0,\infty) \times \Sigma_1 , \\
\label{eq:VCktBd3}
&& (JF)(0,y) =  (JF) (v_F,y)   \ \hbox{ for any } \  y > y_F, \\
\label{eq:VCktBd1}
&& K_F   F - a_F \partial_y F = 0 \ \hbox{ on } \ \Gamma_0 := (0,\infty) \times \Sigma_0 , 
 \eear
 where we define $y_F := y_L v_F/(v_E-v_F)$, 
 \bean
 \Sigma_0 := (0,v_F) \times \{ 0 \}, \quad
  \Sigma_1 = (\{ 0 \} \cup \{v_F \}) \times (0,y_F). 
 \eean
The first condition ensures that there is no net movement of the probability density across the boundary at low conductance values. Note that $y_F$ is the solution of the equation $J(y,v_F)=0$, and more precisely, it is the point where the function $y\mapsto J(y, v_F)$ changes of sign. For values of $y\leq y_F$, the vector field $J$ is negative near $v_F$ so that the trajectories move away from the boundary at $v_F$, i.e., for conductance values smaller than $y_F$,  
the neuron does not receive enough current to spike. On the other hand,  as $y$ pass the value $y_F$ the voltage of the neuron may reach $v_F$, spiking. In that case, it is reset to zero instantaneously, from were the second boundary condition for $y>y_F$. The condition in $\Gamma_0$ follows to assure mass conservation. 
 
 We note $\Sigma_{12} :=  \Sigma_1 \cup \Sigma_2$, $\Sigma_2 := (\{ 0 \} \cup \{v_F \}) \times (y_F,\infty)$, $\Sigma := \Sigma_{12} \cup \Sigma_0$, $\Gamma = (0,\infty) \times \Sigma$, $\Gamma_i = (0,T) \times \Sigma_i$, $i=0,1,2,12$, and $n_v := + 1$ for $v = v_F$, $n_v := - 1$ for $v=0$, 
 $n_0 := -1$. For further references, we also denote 
$$
\Sigma_i^{\pm} := \{(v,y) \in \Sigma_i; \, \pm J n_v > 0 \}, 
$$
so that 
$$
\Sigma_1^+ = \emptyset, \  \Sigma_1^- = \Sigma_1, \ \Sigma_2^- = \{0\} \times (y_F,\infty), \  \Sigma_2^+ = \{v_F\} \times (y_F,\infty),
$$
and we next define similarly $\Sigma_{12}^- = \Sigma_{1}^- \cup \Sigma_2^-$, $\Gamma_i^\pm := (0,T) \times \Sigma_i^\pm$. 
We will sometime summarize \eqref{eq:VCktBd1}-\eqref{eq:VCktBd2}-\eqref{eq:VCktBd3} with the shorthand 
\beqn\label{eq:VCkBd}
\RRR_F \gamma F = 0 \quad\text{on}\quad \Gamma := (0,\infty) \times \Sigma,
 \eeqn
 where $\gamma F$ denotes the trace on $\Gamma$ of the function $F$ in order to recall that $\RRR_F$ acts on the boundary. 

 \medskip
 The two important general properties of the model are that (at least formally) any solution is mass conservative, that is 
\beqn\label{eq:MassCons}
\langle F(t,\cdot) \rangle = \langle F_0 \rangle, \quad \forall \, t \ge 0, \quad \langle F \rangle :=  \int_\OO F dvdy, 
\eeqn
  and it conserves positivity, that is
\beqn\label{eq:PositiveCons}
F(t,\cdot) \ge 0, \quad \forall \, t \ge 0, \quad \hbox{ if } \quad F_0 \ge 0.
\eeqn

\subsection{Main results}

We introduce a class of weight functions that will be useful in the sequel. We say that $\omega : [0,\infty) \to [0,\infty)$ is an admissible weight function if 
$$
  \omega (y) := \langle y \rangle^k, \   k > 1, \quad\hbox{or}\quad
\omega (y) := e^{\alpha y}, \ \alpha > 0, 
$$
with $\langle y \rangle^2 := 1 + y^2$.  
We then define the weighted Lebesgue space $L^p_\omega$ associated to the norm
$$
\| f \|_{L^p_\omega} := \| f \omega \|_{L^p}.
$$

\medskip
On the one hand, we establish the existence of a solution to the nonlinear problem. 

\begin{theo}\label{theo-Exists} 
For any admissible weight function $\omega$, there exists a constant $\eta = \eta(\omega) \in (0,\infty)$ such that for any initial datum $0 \le F_0 \in L^\infty_\omega$ and any connectivity parameter $\cc \ge 0$ such that $(\cc +\cc^2) \| F_0 \|_{L^\infty_\omega} < \eta$, there exists at least one  solution $F \in L^\infty(0,\infty;L^\infty_\omega) \cap C([0,\infty);L^2_\omega)$
to the Voltage-Conductance kinetic (VCk) equation   \eqref{eq:VCk},  \eqref{eq:defNNF}, \eqref{eq:VCkt=0}, \eqref{eq:VCktBd1}, \eqref{eq:VCktBd2},  \eqref{eq:VCktBd3} which furthermore satisfies \eqref{eq:MassCons} and 
\eqref{eq:PositiveCons}. 
\end{theo} 

Inspired by  \cite{MR0153974,MR0875086,MR2721875,sanchez2023kreinrutman,CM-Landau**}, the precise definition of solutions will be given in Section~\ref{sec:ExistenceNL}. 
Our result provides in a weak interaction regime the existence of  a quite strong solution. Under the normalization hypothesis that $F_0$ is a probability measure, so that is $F(t,\cdot)$ for any $t \ge 0$, that corresponds to a weak connectivity regime, namely  $\cc > 0$ small enough.  Another framework of possible  weak  solutions has been developed in \cite[Sec.~6]{MR3177631}. 
The proof  of Theorem~\ref{theo-Exists} is based on the construction of an invariant bounded set for the uniform norm and a compactness argument (Tyknonv fixed point theorem). 

\smallskip

 It has been established in  \cite[Theorem~3]{MR3177631} that in a weak connectivity regime, and more precisely when $0 \le \cc < v_F/v_E$, there exists a probability measure $\MMM$ on $\OO$ such that 
  \bear \label{eq:VCkequilibre1}
 &&  \partial_v(J\MMM) + \partial_y(K_\MMM \MMM) - a_\MMM \partial_{yy}^2 \MMM = 0  \quad\text{on}\quad   \OO, 
 \\ \label{eq:VCkequilibre2}
&& \RRR_\MMM \gamma \MMM = 0 \quad\text{on}\quad   \Sigma, 
 \eear
and that furthermore 
$$
\MMM \in L^q, \quad \forall \, q \in [1,8/7). 
$$
This stationary state has been proved to be uniformly bounded and  linearly asymptotically stable in \cite{MR4548856} and next  linearly  asymptotic exponentially stable in \cite{dou2023voltageconductance}. We recover and slightly improve these results.

\begin{theo}\label{theo-StabLinear} For $\omega_0 := \langle y \rangle^{k_0}$ and  $\cc < \eta(\omega_0) \| \omega_0^{-1} \|_{L^1}$, there exists a probability measure $\MMM$ 
which satisfies the stationary equation \eqref{eq:VCkequilibre1}, \eqref{eq:VCkequilibre2}. 
Furthermore, for any admissible weight function $\omega$, the stationary state $\MMM$ satisfies 
$$
\MMM \in L^\infty_\omega(\OO) \cap C(\OO)  
$$
and it is linearly asymptotic exponentially stable with constructive constants. More precisely, there exist $C \ge 1$, $\lambda > 0$ constructive constants
such that for any $f_0 \in L^2_\omega \cap L^p_\omega$, $p \in [1,\infty]$,  such that $\langle f_0 \rangle = 0$, there exists a unique solution $f \in C([0,\infty);L^2_\omega)$ to the linear equation 
  \bear \label{eq:VCkequilibre-f1}
 && \partial_t f +  \partial_v(J f) + \partial_y(K_\MMM f) - a_\MMM \partial_{yy}^2 f = 0  \quad\text{on}\quad  (0,\infty) \times  \OO, 
 \\ \label{eq:VCkequilibre-f2}
&& \RRR_\MMM \gamma f = 0 \quad\text{on}\quad  (0,\infty) \times  \Sigma, 
 \eear
and this one satisfies 
$$
\| f_t \|_{L^p_\omega} \le C e^{-\lambda t} \| f_0 \|_{L^p_\omega},  
\quad \forall \, t \ge 0.
$$
\end{theo} 

The proof is in the spirit of the recent works \cite[Sec.~11]{sanchez2023kreinrutman}, \cite{carrapatoso2024kinetic} and \cite{CGMM**} where the linear asymptotic exponential stability for the kinetic Fokker-Planck equation set in a domain is established by either taking advantage of hypocoercivity structure (see \cite{MR2562709,MR4581432}) or either taking advantage of  positivity apparence (through mixing ans spreading) and confinement mechanism.
We rather follow that second way and more precisely we establish and use a variant of Doblin-Harris theorem as developed recently in \cite{MR2857021,zbMATH07654553,sanchez2023kreinrutman,CGMM**}.

\smallskip
Both above results use in a fundamental way the ultracontractivity property of solutions to the linear (with possibly time dependent coefficients) VCk equation
 \beqn\label{eq:VCklinear}
\partial_t f  = \LLL_{\fraka,K} f := -  \partial_v(Jf) -  \partial_y(  K f) +  \fraka  \partial_{yy}^2 f
\quad\text{in}\quad (0,\infty) \times \OO, 
\eeqn
where, for some $a^* > \max(a_*,y_*)$, 
\bean
K :=  \frakb - y, \quad  \fraka, \frakb  \in L^\infty(0,\infty), \quad  y_* \le \frakb \le a^*, \quad a_* \le  \fraka \le a^*, 
\eean 
and the evolution equation is complemented with the boundary conditions \eqref{eq:VCktBd2}, \eqref{eq:VCktBd3} and 
\beqn\label{eq:VCktBd1-linear}
 K   f - \fraka  \partial_y f = 0 \ \hbox{ on } \ \Gamma_0.
\eeqn
We will sometime summarize  \eqref{eq:VCktBd2}-\eqref{eq:VCktBd3}-\eqref{eq:VCktBd1-linear} with the shorthand 
\beqn\label{eq:LinearVCkBd}
\RRR_{\fraka,K} \gamma f = 0 \quad\text{on}\quad \Gamma.
 \eeqn

\begin{theo}\label{theo-Ultra} For any exponential weight function $\omega := e^{\alpha y}$, $\alpha > 0$, and any initial datum $f_0 \in L^2_\omega$, the solution $f$ to the linear  Voltage-Conductance kinetic equation  \eqref{eq:VCklinear}, \eqref{eq:LinearVCkBd}
 satisfies the ultracontractivity estimate 
\beqn\label{eq:theo-Ultra}
\| f(t,\cdot) \|_{L^\infty_\omega} \lesssim {  e^{\kappa (t-s)}  \over (t-s)^\nu} \| f(s,\cdot) \|_{L^1_\omega}, 
\eeqn
for any $t > s \ge 0$, some constants   $\nu> 0$ and  $\kappa$ such that $0 \le \kappa \lesssim 1+ \| \fraka \|_{L^\infty} + \| \frakb \|_{L^\infty}$. 
\end{theo} 

Here and below, we use the usual notation $a \lesssim b$ and $b \gtrsim a$ when $a \le Cb$ for some constant $C \in (0,\infty)$ and $a \simeq b$ when $a \lesssim b$ and $b \lesssim a$.

\smallskip

This result is a variant of the ultracontractivity results for the kinetic Fokker-Planck (KFP) equation in a domain established in \cite{carrapatoso2024kinetic,CGMM**}, see also \cite{CM-Landau**}. 
For general parabolic equations in the whole space, ultracontractivity estimates are a consequence of the De Giorgi-Nash-Moser theory  \cite{MR0093649,MR0100158,MR0170091,MR159139}, although the notion of ultracontractive semigroup has been introduced later by Davies and Simon \cite{MR0766493} (see also  \cite{MR0712583}). For the classical KFP equation is the whole space,  ultracontractivity property is a direct consequence of the representation of the solution thanks to the Kolmogorov kernel \cite{MR1503147}, of some regularity estimates through Fourier analysis \cite{MR0222474,MR1949176} or of some modified energy estimates \cite{MR2294477}. 
Some local uniform estimate of a similar kind for a larger class of KFP equations in the whole space has been established in \cite[Theorem 1.2]{MR2068847} by using Moser iterative scheme  and in  \cite[Theorem~2.1]{zbMATH07050183} by using  De Giorgi iterative scheme.  We refer to the above mentioned papers, in particular to \cite{carrapatoso2024kinetic}, for more references on the subject as well as to   \cite{MR4069607} for a general survey about these issues. The proof of Theorem~\ref{theo-Ultra} is similar to the one of \cite[Theorem~1.1]{carrapatoso2024kinetic} although some steps must be modified. 
For the gain of integrability $L^1 \to L^2$, the strategy is based on Nash's gain of integrability argument \cite{MR0100158} which is performed however on the time  integral inequality
as in  Moser's work  \cite{MR159139} what is more convenient in order to use the interior gain of  integrability deduced from an adaptation of H{\'e}rau regularity result  \cite{MR2294477}. 
As in  \cite{carrapatoso2024kinetic}, the key argument consists in exhibiting a suitable  twisted weight function which makes possible to obtain a priori growth estimate in weighted $L^p$ space and a nice control (through a penalization term)  of the boundary region. Last, the gain of  integrability $L^2 \to L^\infty$ is obtained by adapting Moser iterative scheme as in  \cite{MR0170091,MR159139,MR2068847}.

 \medskip 
 The organization of the paper is as follows. 
Section~\ref{sec:GrowthEstim} is dedicated to the proof of some weighted $L^p$ a priori growth bounds and the well-posedness of the linear VCk equation in a weighted $L^2$ framework. 
 Section~\ref{sec:ultra} is devoted to the proof of the ultracontractivity property as stated in Theorem~\ref{theo-Ultra}. 
 In section~\ref{sec:Holder&compact}, we establish some useful interior Holder regularity and up to the boundary compactness results. 
 Section~\ref{sec:ExistenceNL} is dedicated to the proof of the existence result for the nonlinear problem. 
 In Section~\ref{sec:DHtheorem}, we state and prove a constructive version of the Doblin-Harris theorem in a general lattice Banach framework, which is used in Section~\ref{sec:asympt} in order to prove Theorem~\ref{theo-StabLinear}.

\section{Growth estimates and well-posedness in $L^2$} 
\label{sec:GrowthEstim}

 \subsection{Growth estimates on the primal problem}
\label{subsec:GrowthEstimPrimal}

 We establish some a priori  growth estimates in weighted Lebesgue spaces for a solution $f$ to the linear  VCk equation.

\begin{lem}\label{lem:GrowthLp}
For any admissible weight function $\omega = \omega(y)$,
there exist a modified weight function $\widetilde\omega = \widetilde\omega(v,y) \simeq \omega$  
and a constant $ \kappa  \ge 0$ such that any solution  $f$ to the linear  VCk equation \eqref{eq:VCklinear}, \eqref{eq:LinearVCkBd} 
 satisfies (at least formally)  
\beqn\label{eq:lem:GrowthLp}
\| f_t \|_{L^p_{\widetilde\omega}} \le e^{\kappa t } \| f_0 \|_{L^p_{\widetilde\omega}}, \quad \forall \, t \ge 0,
\eeqn
and $\kappa \le C( 1+ \| \fraka \|_{L^\infty} + \| \frakb \|_{L^\infty})$ for some constant $C=C_\omega \in (0,\infty)$.
\end{lem}

\begin{proof}[Proof of Lemma~\ref{lem:GrowthLp}]
For simplicity and because the linear VCk equation is (at least formally) positivity preserving, we may focus on the case   $f = f(t,v,y) \ge 0$.  
For $p \in [1,\infty)$, we  write
\bean
 \frac1p  {d \over dt} \int_\OO f^p \widetilde\omega^p
&=& \int_\OO  \partial_v(-Jf) f^{p-1} \widetilde\omega^p + \int_\OO \partial_y(\fraka\partial_y f - Kf)  f^{p-1}\widetilde\omega^p =: A + B,
\eean
and we then consider each term separately. For the first term, we have 
\bean
A 
&=& - \int_\OO f^p\widetilde\omega^p  \partial_vJ - \frac1p \int_\OO J\widetilde\omega^p \partial_v(f^p)
\\
&=&\int_\OO f^p \bigl( \frac1p \partial_v(\widetilde\omega^pJ) - \widetilde\omega^p  \partial_vJ \bigr)  
-   \frac1p \int_{\Sigma} J f^p \widetilde\omega^p n_v
\\
&=&\int_\OO f^p \widetilde\omega^p  \bigl( J  {\partial_v \widetilde\omega \over \widetilde\omega} 
+ (\frac1p-1) \partial_v J  \bigr)  
-   \frac1p \int_{\Sigma_2} J f^p \widetilde\omega^p n_v,
\eean
 where we have used the Green formula in the second line and the boundary condition on $\Sigma_1$ in the last line. 
On the other hand, we compute
\bean
B
&=& \int_\OO (Kf - \fraka\partial_y f ) \partial_y( (f  \widetilde\omega)^{p-1}) \widetilde\omega)
\\
&=& \int_\OO Kf \widetilde\omega  \partial_y( (f  \widetilde\omega)^{p-1}) ) + \int_\OO Kf  (\partial_y\widetilde\omega)  (f  \widetilde\omega)^{p-1} 
\\
& & - \int_\OO  \fraka\partial_y f   \widetilde\omega  \partial_y( (f  \widetilde\omega)^{p-1}) )  - \int_\OO  \fraka\partial_y f   (\partial_y\widetilde\omega)  (f  \widetilde\omega)^{p-1} 
\\
&=& \int_\OO Kf \widetilde\omega  \partial_y( (f  \widetilde\omega)^{p-1}) ) + \int_\OO Kf  (\partial_y\widetilde\omega)  (f  \widetilde\omega)^{p-1} 
\\
& & - \int_\OO  \fraka\partial_y (f   \widetilde\omega)  \partial_y( (f  \widetilde\omega)^{p-1}) ) + \int_\OO  \fraka f (\partial_y   \widetilde\omega)  \partial_y( (f  \widetilde\omega)^{p-1}) ) 
\\
&& - \int_\OO  \fraka\partial_y (f\widetilde\omega) {  \partial_y\widetilde\omega \over \widetilde\omega}  (f  \widetilde\omega)^{p-1}  
+ \int_\OO    \fraka     (f  \widetilde\omega)^{p} \Bigl( {  \partial_y\widetilde\omega \over \widetilde\omega} \Bigr)^2
\\
&=& 
\int_\OO [(1-\frac1p) K + (1-\frac2p) \fraka    {\partial_y\widetilde\omega \over \widetilde\omega} ]     \partial_y(   (f \widetilde\omega)^p  ) + \int_\OO K {\partial_y\widetilde\omega \over \widetilde\omega}   (f  \widetilde\omega)^{p} 
\\
& & - 4 {p-1 \over p^2} \int_\OO  \fraka (\partial_y (f   \widetilde\omega)^{p/2})^2   + \int_\OO  \fraka (f  \widetilde\omega)^{p} \Bigl( {  \partial_y\widetilde\omega \over \widetilde\omega} \Bigr)^2,
\eean
where we have used the Green formula and  boundary condition on $\Sigma_0$ in the first line and we have then just rearrange the terms  in order to forcing the dependence on the function $f   \widetilde\omega$. 
Using once more the  Green formula in the first integral of the term $B$ and gathering the two contributions, we obtain 
\bean
 \frac1p  {d \over dt} \int_\OO f^p \widetilde\omega^p
&=& - {4 (p-1) \over p^2}  \int_\OO \fraka (\partial_y (f\widetilde\omega)^{p/2})^2 +  \int_\OO  f^p \widetilde\omega^p \varpi
\\
&& + \int_{\Sigma_0} [(1-\frac1p) K + (1-\frac2p) \fraka      {\partial_y\widetilde\omega \over \widetilde\omega} ] n_0     (f \widetilde\omega)^p   
-   \frac1p \int_{\Sigma_2} J f^p \widetilde\omega^p n_v,
\eean
with 
\beqn\label{def:varpi}
\varpi := 2(1-\frac1p) \fraka \bigl( {\partial_y \widetilde\omega \over \widetilde\omega} \bigr)^2
+ (\frac2p - 1) \fraka {\partial^2_{yy} \widetilde \omega \over \widetilde\omega}
+ K  {\partial_y \widetilde\omega \over \widetilde\omega} + J  {\partial_v \widetilde\omega \over \widetilde\omega} 
+ (\frac1p-1) \partial_y K 
+ (\frac1p-1) \partial_v J. 
\eeqn
In order to make negative the boundary contribution, we define first 
\beqn\label{eq:def-w}
w := \chi + (1-\chi) \omega, 
\eeqn
with $\chi = \chi(y) \in C^2(\R)$, ${\bf 1}_{[0,y_F/2]} \le \chi \le {\bf 1}_{[0,y_F]}$,  and next 
\beqn\label{eq:def-widetildeomega}
  \widetilde \omega^p := (\chi + (1-\chi) J_\xi ^{p-1}y^{1-p}) w^p, 
  \quad  
J_\xi :=   \xi J(0,y)  +  (1- \xi)J(v,y), 
\eeqn
with $\xi = \xi(y) \in C^2(\R)$, ${\bf 1}_{[0,y_F]} \le \chi \le {\bf 1}_{[0,2y_F]}$. 
 
\smallskip
With this choice, we have $\widetilde \omega =  1$  
on $(0,v_F) \times (0,y_F/2)$, thus   
$$
(1-\frac1p) K + (1-\frac2p) \fraka      {\partial_y\widetilde\omega \over \widetilde\omega} = (1-\frac1p) K \ge (1-\frac1p) y_* \ge 0 \quad \hbox{ on } \quad \Sigma_0, 
$$
and the  contribution of the boundary term on $\Sigma_0$ is non positive.  
On the other hand, the contribution of the boundary term on $\Gamma_2$ is 
\bean
 - \frac1p \int_{\Sigma_2} J f^p \widetilde\omega^p n_v
&=& \frac1p \int_{y_F}^{\infty} (J f^p  \widetilde\omega^p) (0,y) -  (J f^p  \widetilde\omega^p) (v_F,y)
 \\
&=& \frac1p \int_{y_F}^{\infty} ((J^p f^p  ) (0,y) -   (J J_\xi^{p-1} f^p ) (v_F,y)) w^p  y^{1-p} 
 \\
&\le& \frac1p \int_{y_F}^{\infty} ((Jf )^p (0,y) -  (Jf)^p (v_F,y))   w^2 y^{1-p} = 0,
 \eean
 where we use that $ J_\xi  \ge J$ in the second line and the boundary condition in the last line.  
 Now, from the very definition of $\widetilde\omega$,  we have 
 $$
  \varpi \le \kappa_1 := C (1 + \| \fraka \|_{L^\infty} + \| \frakb \|_{L^\infty}) \quad\hbox{on}\quad (0,T) \times (0,v_F) \times (0,2y_F]
 $$
 uniformly on $p \in [1,\infty)$ and for some constant $C = C(\widetilde\omega) \ge 0$.  
 On the other hand, 
we have $\widetilde \omega = Q^{1-1/p} \omega$ on  $(0,v_F) \times (2y_F,\infty)$, $Q := J/y = v_E - v  - y_Lv/y$, so that  
\bean
{\partial_v \widetilde \omega \over  \widetilde \omega}
&=& (1-\frac1p) {\partial_vJ  \over  J }  
 \\
{\partial_y \widetilde \omega \over  \widetilde \omega}
&=&
  (1-\frac1p) {\partial_y Q  \over  Q }
+  {\partial_y\omega \over \omega} =  {\partial_y\omega \over \omega} + \OO(\langle y \rangle^{-1}) 
 \\
{\partial^2_{yy} \widetilde \omega \over  \widetilde \omega}
&=&
- \frac1p(1-\frac1p) \bigl( {\partial_y Q \over Q} \bigr)^2 + (1-\frac1p) {\partial_{yy}^2 Q \over Q}  + 2 (1-\frac1p) {\partial_{y} Q \over Q}  {\partial_{y} w \over w}   + 
 {\partial_{yy}^2 w \over w} . 
 \eean  
 Observing that 
 $$
 Q =  v_E-v - y_L v/y, \quad  
 {\partial_y Q  \over  Q } = \OO(y^{-2}), \quad 
 {\partial^2_{yy}  Q  \over  Q } = \OO(y^{-3}), \quad  
 {\partial_v Q  \over  Q } = - {1 \over v_E- v} + \OO(y^{-1}),  
  $$
because $\omega$ is an admissible weight function, 
 we deduce 
 \bean
{\partial_y \widetilde \omega \over  \widetilde \omega}
\sim
  {\partial_y\omega \over \omega}, 
\quad
{\partial^2_{yy} \widetilde \omega \over  \widetilde \omega}
\sim   {\partial_{yy}^2 w \over w}
 \eean
and next 
\bean
\varpi 
&=& 2(1-\frac1p) \fraka  \bigl( {\partial_y \widetilde\omega \over \widetilde\omega} \bigr)^2
+ (\frac2p - 1)  \fraka  {\partial^2_{yy} \widetilde \omega \over \widetilde\omega}
+ K  {\partial_y \widetilde\omega \over \widetilde\omega}  
+ 1 - \frac1p 
\\
&\sim& 2(1-\frac1p) \fraka \bigl( {\partial_y  \omega \over \omega} \bigr)^2
+ (\frac2p - 1) \fraka {\partial^2_{yy}   \omega \over  \omega}
- y  {\partial_y  \omega \over  \omega}  
+ 1 - \frac1p 
\eean
as $y\to\infty$. 
When $\omega = y^k$,   we have 
\beqn\label{eq:asym-varpi-poly}
\varpi 
\sim  - k  + 1 - \frac1p.
\eeqn
When $\omega = e^{\alpha y}$, $\alpha > 0$, we have 
\bean
\varpi 
 &\sim&  - \alpha y. 
\eean
  More precisely, in both cases, we have 
\beqn\label{eq:asym-varpi-expo}
\varpi  \le \kappa_2 - \varsigma\quad\hbox{on}\quad (0,T) \times (0,v_F) \times (2y_F,\infty)
\eeqn
 uniformly on $p \in [1,\infty)$, with 
 \bean
 && \kappa_2 := C (1 + \| \fraka \|_{L^\infty} + \| \frakb \|_{L^\infty}) 
 \\
 && \varsigma := k-1+\frac1p \hbox{ if } \omega = \langle y \rangle^k, 
\quad \varsigma := \alpha y \hbox{ if } \omega = e^{\alpha y }. 
\eean

 \smallskip
 All together, we have established 
\beqn\label{eq:LpestimGrowth-1}
 \frac1p  {d \over dt} \int_\OO f^p \widetilde\omega^p
\le - {4 (p-1) \over p^2}  \int_\OO \fraka (\partial_y (f\widetilde\omega)^{p/2})^2 +  \int_\OO  f^p \widetilde\omega^p \varpi, 
\eeqn
and from \eqref{eq:asym-varpi-poly} and  \eqref{eq:asym-varpi-expo}, we deduce 
\beqn\label{eq:LpestimGrowth-1}
 \frac1p  {d \over dt} \int_\OO f^p \widetilde\omega^p
\le - {4 (p-1) \over p^2}  \int_\OO \fraka (\partial_y (f\widetilde\omega)^{p/2})^2 + \kappa \int_\OO  f^p \widetilde\omega^p , 
\eeqn
with $\kappa := \max(\kappa_1,\kappa_2)$. We conclude thanks to the Gronwall lemma in the case $p \in [1,\infty)$ and passing to the limit $p \to \infty$ in the resulting estimate for dealing with the case $p=\infty$.  
 \end{proof}

\medskip

 \subsection{About the well posedness}
 
 In this section, we build a solution to the Cauchy problem associated to the linear VCk equation in a $L^2$ framework. 
We denote $d\xi_1 := \widetilde\omega^2  |J| dtdy$  and $d\xi_2 := \omega^2  J^2 /\langle y \rangle^2 dtdy$  the Borel measures on the boundary $\Gamma_{12}$ and also 
$d\xi_1 := \widetilde\omega^2  K dtdv$  the Borel measures on the boundary $\Gamma_{0}$.

\begin{theo}\label{theo:existL2}
For any admissible weight function $\omega$ and any initial datum $f_0 \in L^2_\omega \subset L^1$, there exists a unique solution $f \in C([0,T];L^2_\omega) \cap L^2((0,T) \times (0,v_F); H^1(0,\infty))$, $\forall \, T > 0$,  to the linear  VCk equation \eqref{eq:VCklinear}-\eqref{eq:LinearVCkBd} 
 and this one satisfies the growth estimate \eqref{eq:lem:GrowthLp} (for $p=2$). 
More precisely, there exists a trace function $\gamma f \in L^2(\Gamma_{12}; d\xi_2)$ such that $f$ is a renormalized solution, in the sense that 
\bear
\label{eq:theo:existL2-1}
&&  \int_\OO \beta(f_T) \psi + \int_{\Gamma} J  \beta(\gamma f)  \psi  + \int_\UU \beta(f) [ \partial_t \psi +  \partial_v (J \psi)]  - \int_\UU  (\partial_vJ)f \beta'(f) \psi 
\\ \nonumber
&&\quad = -  \int_\UU  (K f - \fraka \partial_{y} f ) \partial_y  (\beta'(f)\psi)   +  \int_\OO \beta(f_0) \psi ,
\eear
for any $\beta \in W^{1,\infty}(\R) $ and any $\psi \in C^1_c(\bar\UU)$,  and where $\gamma f$ satisfies the  boundary conditions \eqref{eq:VCktBd2}, \eqref{eq:VCktBd3} in the a.e. sense. 
The additional  boundary condition \eqref{eq:VCktBd1-linear} is encapsulated in the fact that $\psi$ does not necessarily  vanish on the boundary set $\Gamma_0$. 
Furthermore, if $f_0 \ge 0$, the solution $f$ satisfies 
$$
f(t,\cdot) \ge 0  \quad\hbox{and}\quad \| f(t,\cdot) \|_{L^1} = \| f_0 \|_{L^1}, \quad \forall \, t \ge 0. 
$$
\end{theo}

\begin{proof}[Proof of Theorem~\ref{theo:existL2}] The proof is a bit tedious but classical. It is  split into 7 steps. 
 
\smallskip\noindent
\textit{Step 1.} We define $\Gamma_{2}^- := \{ (t,v,y) \in \Gamma_2,  J n_v < 0 \} = (0,T) \times \{ 0\} \times (y_F,\infty)$.
Given a function $\mathfrak g \in L^2(\Gamma^-_{2}; J \omega^2 dtdy)$, we solve the inflow problem \eqref{eq:VCklinear},  \eqref{eq:VCktBd1-linear}, 
\eqref{eq:VCktBd2} and 
\beqn  \label{eq:LionsBd2bis}
  f    =  {\frak g}  \ \hbox{ on  } \  \Gamma^-_2
\eeqn
 thanks to Lions' variant of the Lax-Milgram theorem \cite[Chap~III,  \textsection 1]{MR0153974}. 
More precisely, we define the Hilbert space $\HHH$ associated to 
the Hilbert norm $\| \cdot \|_\HHH$ defined by 
$$
  \| f \|^2_\HHH := \| f \|^2_{L^2_{\omega_1} (\UU)} +  \|  \partial_y f  \|^2_{L^2_\omega(\UU)}, 
$$
with $\omega_1 = \omega$ in the polynomial case and $\omega_1 = \omega \langle y \rangle^{1/2}$ in the exponential case, 
and we define  the bilinear form
$\EEE_\lambda : \HHH \times C_c^1([0,T) \times (\OO \cup \Sigma_-)) \to \R$,  by  
\bean
\EEE_\lambda(f,\varphi) 
:= \int_\UU   f (\lambda   - \partial_t  - J \partial_v - K \partial_y ) (\varphi \widetilde\omega^2) + \int_\UU \fraka \partial_y f  \partial_y  (\varphi \widetilde\omega^2), 
\eean
where $\lambda > 0$,  $\Sigma_- : =  \Sigma_0 \cup \Sigma_1 \cup \Sigma_{2-}$ 
and $\widetilde \omega$ is defined during the proof of Lemma~\ref{lem:GrowthLp}. 
Proceeding as during the proof of Lemma~\ref{lem:GrowthLp}, we have 
\bean
 \int_\UU \varphi (- J) \partial_v   (\widetilde\omega^2 \varphi)
 &=& \int_\UU   (\widetilde\omega \varphi)^2 \bigl( \frac12 \partial_v J - J {\partial_v \widetilde \omega \over \widetilde \omega}) -  \frac12 \partial_v  (J  (\widetilde\omega \varphi)^2) 
\\
 &=& - \int_{\Gamma_{12}} \frac12 n_v  J  (\widetilde\omega \varphi)^2
 + \int_\UU 
    (\widetilde\omega \varphi)^2 \bigl( \frac12 \partial_v J - J {\partial_v \widetilde \omega \over \widetilde \omega}) 
 \eean
 and similarly 
 \bean
 \int_\UU \varphi (- K) \partial_y   (\widetilde\omega^2 \varphi)
 &=& - \int_{\Gamma_{0}} \frac12 n_y  K  (\widetilde\omega \varphi)^2
 + \int_\UU 
    (\widetilde\omega \varphi)^2 \bigl( \frac12 \partial_y  K -  K {\partial_y \widetilde \omega \over \widetilde \omega}). 
 \eean
Observing that 
$$
\int_\UU \fraka \partial_y \varphi \partial_y  (\varphi \widetilde\omega^2) = \int_\UU \fraka (\partial_y  (\varphi \widetilde\omega))^2 - 
\int_\UU \fraka (\varphi \widetilde\omega)^2 { (\partial_y  \widetilde\omega)^2 \over  \widetilde\omega^2}, 
$$
we get 
\bean
\EEE_\lambda(\varphi,\varphi) 
&=& \int_\UU (\lambda - \varpi) (\widetilde\omega \varphi)^2 +  \int_\UU \fraka (\partial_y (\widetilde\omega \varphi))^2  +  \frac12 \int_\OO (\widetilde\omega \varphi)(0,\cdot)^2   
\\
&&- \frac12 \int_{\Gamma_{12}}  (\widetilde\omega \varphi)^2  J n_v - \frac12  \int_{\Gamma_{0}} (\widetilde\omega \varphi)^2   K n_y ,
\eean
for any $\varphi \in C_c^1(\UU \cup \Gamma_-)$, and where $\varpi$ is defined (for $p=2$) in \eqref{def:varpi}. 
We choose $\lambda > 0$ large enough, in such a way that $\lambda - \varpi \ge 1$. 
Because each contribution is then nonnegative separately, the above quadratic form is coercive in the sense that 
$$
\EEE_\lambda(\varphi,\varphi) \ge \min(a_*,1) \| \varphi \|^2_\HHH, \quad \forall \, \varphi \in  C_c^1(\UU \cup \Gamma_-).
$$
We define the linear form 
$$
\ell_\lambda(\varphi) := 
\int_{\Gamma_{2-}} (-J n_v )\mathfrak{g} e^{-\lambda t} \widetilde\omega^2 \varphi 
+ \int_\OO f_0\widetilde\omega^2 \varphi(0,\cdot),  
\quad \forall \, \varphi \in C_c^1(\UU \cup \Gamma_-), 
$$
and we observe that $|\ell(\varphi)| \lesssim \EEE_\lambda (\varphi,\varphi)^{1/2}$ for any $ \varphi \in C_c^1(\UU \cup \Gamma_-)$. 
The above mentioned Lions' theorem implies the existence of a function $f_\lambda \in \HHH$
which satisfies the variational equation 
$$
\EEE_\lambda(f_\lambda,\psi) = \int_{\Gamma_{2-}} \mathfrak{g} e^{-\lambda t} \widetilde\omega^2 \psi  
+ \int_\OO f_0\widetilde\omega^2 \psi(0,\cdot), 
\quad \forall \, \psi \in C_c^1(\UU \cup \Gamma_-). 
$$
Defining $f := f_\lambda e^{\lambda t}$ and testing this variational equation with $\psi :=  \varphi e^{\lambda t}$, we deduce 
\beqn\label{eq:varTOrenorm-0}
 \int_\UU   f (- \partial_t  - J \partial_v)  (\varphi \widetilde\omega^2) + \int_\UU (\fraka \partial_y f  - K   )  \partial_y  (\varphi \widetilde\omega^2)
= \int_{\Gamma_{2-}}  (-J n_v ) \mathfrak{g}   \varphi  \widetilde\omega^2
+ \int_\OO f_0 \varphi(0,\cdot) \widetilde\omega^2,  
\eeqn
for any $ \varphi \in C_c^1(\UU \cup \Gamma_-)$, and thus for any $ \varphi \in H^1_0(\UU \cup \Gamma_-)$, the closure of $C_c^1(\UU \cup \Gamma_-)$ in $H^1(\UU)$. 
That last variational equation is classically a weak formulation of the fact that $f \in \HHH_T$, $\forall \, T > 0$,  is a global solution
to the inflow problem \eqref{eq:VCklinear},  \eqref{eq:VCktBd1-linear}, \eqref{eq:VCktBd2}, \eqref{eq:LionsBd2bis}. It is worth emphasizing that here and always below the no flux condition  \eqref{eq:VCktBd1-linear} has to be understood in the weak sense, namely by the fact that the test functions do not necessarily vanish on $\Gamma_0$.

\smallskip\noindent
\textit{Step 2.}  At least formally, multiplying  the equation \eqref{eq:VCklinear} by $ \beta'(f) \psi$ for a renormalizing function $\beta : \R \to \R$ and a test function $\psi : \UU \to \R$, we have 
\bean
&&\partial_t (\beta(f) \psi) - \beta(f) \partial_t \psi + \partial_v [ J \beta(f) \psi ]  - \beta(f)   \partial_v (J \psi) + (\partial_vJ)f \beta'(f) \psi 
\\
&&\quad   + \partial_y [  \beta'(f) \psi (K f - \fraka \partial_{y} f) ]  -  [\partial_y  (\beta'(f) \psi)] (K f - \fraka \partial_{y} f )  = 0.
\eean
When supp$\psi \subset \UU \cup \Gamma_0 = (0,T) \times (0,v_F) \times [0,\infty)$ and using the no flux condition  \eqref{eq:VCktBd1-linear}, we get 
after integration 
\beqn\label{eq:varTOrenorm-1}
  - \int_\UU \beta(f) [ \partial_t \psi +  \partial_v (J \psi)] 
+ \int_\UU  (\partial_vJ)f \beta'(f) \psi 
 -  \int_\UU  (K f - \fraka \partial_{y} f ) \partial_y  (\beta'(f)\psi)   = 0 .
\eeqn
This step is devoted to sketch a rigorous proof of  \eqref{eq:varTOrenorm-1} by using some classical arguments in the spirit of DiPerna and Lions \cite{MR1022305}. 
For $\beta \in C^1 \cap W^{1,\infty}$, $\psi \in C^1_c( \UU \cup \Gamma_0)$ and a symmetric mollifier $(\rho_\eps)$ in $\DD(\R^2)$, we choose $\varphi_\eps := \widetilde\omega^{-2} \rho_\eps * ( \beta'(f * \rho_\eps ) \psi)$ as a test function in \eqref{eq:varTOrenorm-0}, where $* = *_{t,v}$ denotes the convolution operator in the $t$ and $v$ variables. 
Denoting $f_\eps = f * \rho_\eps$, for $\eps > 0$ small enough, we have 
\beqn\label{eq:varTOrenorm-1eq1}
  \int_\UU   f (- \partial_t  - J \partial_v) (\varphi_\eps \widetilde\omega^2)
= -   \int_\UU    f _\eps    \partial_t ( \beta'(f_\eps) \psi) 
-  \int_\UU   (f J) * \rho_\eps  \partial_v ( \beta'(f_\eps) \psi). 
\eeqn
Observing that $f_\eps \in H^1_0(\UU \cup \Gamma_0)$, so that we may used the chain rule, we  have 
\bean
-   \int_\UU    f _\eps    \partial_t ( \beta'(f_\eps) \psi) 
 &=& \int_\UU  \partial_t  f _\eps   \beta'(f_\eps) \psi 
 \\
  &=& \int_\UU  \partial_t    \beta(f_\eps) \psi  = -  \int_\UU     \beta(f_\eps) \partial_t  \psi. 
\eean
 We similarly have 
 \bean
 -  \int_\UU   (f J) * \rho_\eps  \partial_v ( \beta'(f_\eps) \psi)
= - 
  \int_\UU     \beta(f_\eps)   \partial_v (J\psi) +   \int_\UU  f_\eps  \beta'(f_\eps)  \psi  \partial_v J  
  + R_\eps, 
\eean
with 
$$
R_\eps := \int_\UU  r_\eps \beta'(f_\eps) \psi, 
\quad r_\eps   :=   \partial_v [ (fJ)*\rho_\eps - J (f*\rho_\eps)].
$$
For the last term, we write $r_\eps = r^1_\eps + r^2_\eps$, with 
$$
r^1_\eps := (f \partial_v J) * \rho_\eps - (\partial_v J) f_\eps, 
\quad
r^2_\eps := (J \partial_v f )*\rho_\eps - J \partial_v f_\eps,
$$
where $r^1_\eps \to 0$ in $L^2$ straightforwardly and $r^2_\eps \to 0$ in $L^2$ thanks to DiPerna-Lions commutator Lemma  \cite[Lemma~II.1]{MR1022305}. 
We easily pass to the limit in \eqref{eq:varTOrenorm-1eq1} and we get 
$$
\lim_{\eps\to0}  \int_\UU   f (- \partial_t  - J \partial_v) (\varphi_\eps \widetilde\omega^2)
=  -  \int_\UU     \beta(f) \partial_t  \psi
-  \int_\UU     \beta(f)  \partial_v (J \psi )+ 
   \int_\UU   f   \beta'(f) (\partial_v J)  \psi . 
$$
From the above definition and the fact that $\beta'(f) \in L^2_v L^2_{t,\rm loc} H^1_{y,\rm loc}$, we also have 
\bean
\lim_{\eps \to 0}  \int_\UU (\fraka \partial_y f  - K   )  \partial_y  (\varphi_\eps \widetilde\omega^2)
=    -  \int_\UU  (K f - \fraka \partial_{y} f ) \partial_y  (\beta'(f)\psi).
\eean
All together, we have thus established \eqref{eq:varTOrenorm-1}. 

\smallskip\noindent
\textit{Step 3.} We assume furthermore that $f$ is smooth up to the boundary and we establish three additional estimates. 
Taking $\beta(s) := s^2$ and $\psi = \chi_1(t) \chi_2(v)$, $\chi_1 \in \DD((0,T))$, $\chi_2 \in \DD((0,v_F))$, $\chi_i \ge 0$, in \eqref{eq:varTOrenorm-1}, we get 
$$
{d \over dt} \int_\OO f^2 \chi_2 = \int_\OO I_f, \quad I_f  := 
  f^2[   J \partial_v \chi_2 - \chi_2 \partial_vJ ]
+  2 (K f - \fraka \partial_{y} f ) (\partial_y f ) \chi_2  , 
$$
and integrating twice in the time variable, we obtain
\beqn\label{eq:AprioriBord1}
\sup_{[0,T]}  \int_\OO f^2 \chi_2 \le  T \int_\UU I_f 
 \lesssim  T \| \chi_2 \|_{W^{1,\infty}} \| f \|^2_\HHH. 
\eeqn
Taking now $\psi := \chi_1(t) \chi_2(v)$, $\chi_2 := (v_F/2-v)_+$,   in the equation preceding \eqref{eq:varTOrenorm-1} or equivalently taking $\psi :=  \chi_1(t)  \chi_2(v) \chi_{2\eps}(v)$ in  \eqref{eq:varTOrenorm-1} with $\chi_{2\eps} \to 1$, $\chi_{2\eps}' \wto - \delta_0$ and passing to the limit $\eps \to 0$, we get 
$$
   \int_{\Gamma_{12}^0} f^2  J  \chi_1{v_F \over 2}  = 
   \int_\UU f^2 [ \partial_t \psi+ J\partial_v\psi - \psi \partial_v J] 
   +2 \int_\UU \psi  (K f - \fraka \partial_{y} f ) \partial_y f ,
$$
with $\Gamma_{12}^{0} := (0,T) \times \Sigma_{12}^{0}$, $ \Sigma_{12}^{0} := \{0 \} \times (0,\infty)$.
We deduce 
\beqn\label{eq:AprioriBord2}
  \int_{\Gamma_{12}^0} f^2  J  \chi_1 
 \lesssim     \| \chi_1 \|_{W^{1,\infty}} \| f \|^2_\HHH. 
\eeqn
We finally take $\psi :=  \chi_1(t) \chi_2(v)$, $\chi_2 := J (v-v_F/2)_+$, in the equation preceding \eqref{eq:varTOrenorm-1} and  we get 
$$
   \int_{\Gamma_{12}^{v_F}} f^2  J^2  \chi_1{v_F \over 2}  = 
   \int_\UU f^2 [ - \partial_t \psi -J\partial_v\psi+ \psi \partial_v J ] 
   +2 \int_\UU \psi  (\fraka \partial_{y} f  - K f) \partial_y f ,
$$
with $\Gamma_{12}^{v_F} := (0,T) \times \Sigma_{12}^{v_F}$, $ \Sigma_{12}^{v_F} := \{ v_F \} \times (0,\infty)$. We deduce 
\beqn\label{eq:AprioriBord3}
  \int_{\Gamma_{12}^{v_F}} f^2  J^2  \chi_1 
 \lesssim     \| \chi_1 \|_{W^{1,\infty}} \| f \|^2_\HHH. 
\eeqn

\smallskip\noindent
\textit{Step 4.} 
 Using the same suitably modified convolution trick   as in  the trace theory developed in  \cite{MR1765137,MR2721875,sanchez2023kreinrutman,CM-Landau**} for the Vlasov equation and the kinetic Fokker-Planck equation, we can take up again the arguments used in Step 2 and 
 we easily establish the existence of  a sequence $(f_\eps)$ of $H^1(\UU)$ such that $f_\eps \to f$  strongly in $\HHH$ and a.e. on $\UU$, and such that 
 \bean
  - \int_\UU \beta(f_\eps)[ \partial_t \psi +  \partial_v (J \psi)] 
+ \int_\UU  (\partial_vJ)f_\eps \beta'(f_\eps) \psi 
 -  \int_\UU  (K f - \fraka \partial_{y} f ) \partial_y  \varphi_\eps   = \int_\UU r_\eps \beta'(f_\eps) \psi, 
\eean
for any $\eps>0$, any $\psi \in C^1_c(\UU)$ and any $\beta \in C^1 \cap W^{1,\infty}$, with $r_\eps \to 0$ in $\Lloc^1(\bar\UU)$ and $\varphi_\eps \to \beta'(f) \psi$ in $\HHH$. 
The above relation is the same as the one leading to \eqref{eq:varTOrenorm-1} except on the fact that we do not impose any smallness condition on $\eps > 0$ which should depend on $\psi$.  As a consequence, for $\psi \in C^1_c(\bar\UU)$, we may deduce from it the Green formula which tells us that 
\bear
\label{eq:renormf-eps2}
&&   \Bigl[\int_\OO \beta(f_{\eps t}) \psi \Bigr]_{t_1}^{t_2}+ \int_{t_1}^{t_2}\!\!\int_{\Sigma_{12}} J  h_\eps \psi 
- \int_{t_1}^{t_2}\!\! \int_\OO [\beta(f_\eps) ( \partial_t \psi +  \partial_v (J \psi)) -   (\partial_vJ)f_\eps \beta'(f_\eps) \psi ] 
\\ \nonumber 
&&\quad -  \int_{t_1}^{t_2}\!\! \int_\OO  (K f - \fraka \partial_{y} f ) \partial_y  \varphi_\eps  = \int_{t_1}^{t_2}\!\! \int_\OO r_\eps \beta'(f_\eps) \psi. 
\eear
The same identity holds with $f_\eps$ replaced by $f_{\eps\eps'} := f_\eps - f_{\eps'}$ because this identity has been established starting from the weak formulation \eqref{eq:varTOrenorm-0} which depends linearly on $f$.
Using \eqref{eq:AprioriBord1},  \eqref{eq:AprioriBord2} and  \eqref{eq:AprioriBord3} in Step~3, 
we deduce that $(f_{\eps\eps'})$ tends to $0$ in $C([0,T];\Lloc^2(\OO))$, 
in $L^2((\tau,T-\tau) \times \Sigma^0_{12};Jdydt)$, $\forall \, \tau \in (0,T)$,  
and in  $L^2((\tau,T-\tau) \times \Sigma^{v_F}_{12}; J^2dydt)$, $\forall \, \tau \in (0,T)$,  
so that $(f_\eps)$ is a Cauchy sequence in the same spaces. 
In other words, there exist a function $t \mapsto f_t \in C([0,T];\Lloc^2(\OO))$, and a function $\gamma f$ defined on $\Gamma_{12}$ satisfying 
$\gamma f \in L^2((\tau,T-\tau) \times \Sigma^0_{12}; J dydt)$, $\forall \, \tau \in (0,T)$, 
 and  $\gamma f \in L^2((\tau,T-\tau) \times \Sigma^{v_F}_{12}; J^2dydt)$, $\forall \, \tau \in (0,T)$,  such that 
\bean
&&f_\eps (t, \cdot) \to f_t \hbox{ in }  C([0,T];\Lloc^2(\OO)) \hbox{ and a.e. on } \OO, \ \forall \, t \in [0,T]; 
\\
&&f_\eps  
\to \gamma f  \hbox{ in } L^2((\tau,T-\tau) \times \Sigma^0_{12}; J dydt), \forall \, \tau \in (0,T), \hbox{ and a.e. on } \Gamma_{12}^0; 
\\
&&f_\eps  
\to \gamma f  \hbox{ in } L^2((\tau,T-\tau) \times \Sigma^{v_F}_{12}; J^2dydt), \forall \, \tau \in (0,T), \hbox{ and a.e. on }  \Gamma_{12}^{v_F}.
\eean
Passing to the limit in \eqref{eq:renormf-eps2}, we get 
\bear
\label{eq:renormf-2}
&&   -  \int_{t_1}^{t_2}\!\! \int_\OO [\beta(f) ( \partial_t \psi +  \partial_v (J \psi)) -   (\partial_vJ)f \beta'(f) \psi ]  -  \int_{t_1}^{t_2}\!\! \int_\OO  (K f - \fraka \partial_{y} f ) \partial_y  (\beta'(f)\psi)
\\ \nonumber 
&&\quad +  \Bigl[\int_\OO \beta(f(t,\cdot)) \psi \Bigr]_{t_1}^{t_2}+ \int_{t_1}^{t_2}\!\!\int_{\Sigma_{12}} J n_v \beta(\gamma f) \psi =0 , 
\eear
for any $\psi \in C^1_c(\bar\UU)$ and any $\beta \in C^1 \cap W^{1,\infty}$, and we use indifferently the notations $f_t = f(t,\cdot)$.  
Because $f \in \HHH$, the function $f$ admits a trace on $\Gamma_0$, also denoted by $\gamma f \in L^2(\Gamma_0;dtdv)$, and integrating  twice by part, we have 
\bean
&& -  \int_{t_1}^{t_2}\!\! \int_\OO  K f \partial_y  (\beta'(f)\psi)
=  -  \int_{t_1}^{t_2}\!\! \int_{\Sigma_0}  n_y K \gamma f \beta'(\gamma f)\psi 
+   \int_{t_1}^{t_2}\!\! \int_\OO \beta'(f)\psi \partial_y  ( K f )
\\
&&\qquad =    \int_{t_1}^{t_2}\!\! \int_{\Sigma_0}    K \gamma f \beta'(\gamma f)\psi 
+   \int_{t_1}^{t_2}\!\! \int_\OO [ \beta'(f) f \psi \partial_y K  + \psi K \partial_y \beta(f) ]
\\
&&\qquad =   \int_{t_1}^{t_2}\!\! \int_{\Sigma_0}    K (\gamma f \beta'(\gamma f) - \beta(\gamma f))\psi 
+   \int_{t_1}^{t_2}\!\! \int_\OO [ \beta'(f) f \psi \partial_y K  - \beta(f) \partial_y(\psi K)  ].
\eean
Together with \eqref{eq:renormf-2}, we also have 
\bear
\label{eq:renormf-3}
&&
 \Bigl[\int_\OO \beta(f(t,\cdot)) \psi \Bigr]_{t_1}^{t_2}+ \int_{t_1}^{t_2}\!\!\int_{\Sigma_{12}} J n_v \beta(\gamma f) \psi 
 +  \int_{t_1}^{t_2}\!\! \int_{\Sigma_0}    K [ \gamma f \beta'(\gamma f) - \beta(\gamma f) ] \psi 
\\ \nonumber 
&&   -  \int_{t_1}^{t_2}\!\! \int_\OO [\beta(f) ( \partial_t \psi +  \partial_v (J \psi) + \partial_y(\psi K) ) -   (\partial_vJ + \partial_y K)f \beta'(f) \psi ] 
\\ \nonumber 
&&\quad +  \int_{t_1}^{t_2}\!\! \int_\OO    \fraka \partial_{y} f  \partial_y  (\beta'(f)\psi) =0 , 
\eear
for any $\psi \in C^1_c(\bar\UU)$ and any $\beta \in C^1 \cap W^{1,\infty}$.

 \smallskip\noindent
\textit{Step 5.} Because of the previous estimates on $f$, $f_t$ and $\gamma f$, we may take $\beta(s) = s$ in \eqref{eq:renormf-2} by using an approximation procedure in the three following cases (and more precisely, we apply  \eqref{eq:renormf-2} to a renormalized function $\beta_n \in C^1 \cap W^{1,\infty}$, we assume that $\beta_n(s) \to s$, $\beta_n'(s) \to 1$, with $|\beta_n(s)| \le |s|$, $|\beta_n'(s)| \le 1$, and we pass to the limit as $n\to\infty$). 
Taking $\psi = \chi_1(t) \chi_2$, $\chi_1 \in C^1_c([0,T))$, $\chi_2 \in C^1_c(\OO)$, we have 
\bean
   - \int_\OO  f(0,\cdot) \psi (0,\cdot) - \int_\UU f  ( \partial_t \psi +  J \partial_v \psi )
 -   \int_\UU  (K f - \fraka \partial_{y} f ) \partial_y  \psi   = 0. 
\eean
Together with \eqref{eq:varTOrenorm-0}, we deduce that $f(0,\cdot) = f_0$. Taking now $\psi \in C^1_c((0,T) \times (\OO \cup \Sigma_{12}^-))$, we also have 
\bean
 \int_{\Gamma^-_{12}} J n_v   \gamma f \psi - \int_\UU  f ( \partial_t \psi + J  \partial_v  \psi )
 -  \int_\UU  (K f - \fraka \partial_{y} f ) \partial_y  \psi   = 0, 
\eean
Together with \eqref{eq:varTOrenorm-0}, we deduce that $\gamma f = 0$ on $\Gamma_1$ and $\gamma f = \frak g$ on $\Gamma_2^-$.

\smallskip
Because of the above  identification of the trace functions, the renormalized equation \eqref{eq:renormf-2} writes now
\bear
\label{eq:renormf-4}
&&  \int_\OO \beta(f_T) \psi + \int_{\Gamma_{2}^+} J   \beta(\gamma f)  \psi  = 
 \int_\UU \beta(f) [ \partial_t \psi +  \partial_v (J \psi)]   - \int_\UU  (\partial_vJ)f \beta'(f) \psi 
\\ \nonumber
&&\quad +   \int_\UU  (K f - \fraka \partial_{y} f ) \partial_y  (\beta'(f)\psi)   + 
  \int_{\Gamma^-_{2}} J  \beta(\frak g)  \psi +  \int_\OO \beta(f_0) \psi ,
\eear
for any $\beta \in C^1 \cap W^{1,\infty}$ and $\psi \in C^1_c(\bar\UU)$.
Using that $f \in \HHH$, $f_0 \in L^2_\omega$ and $J^{1/2} \frak g \in L^2_\omega(\Gamma_2^-)$ at the RHS, 
we may take $\psi = \widetilde \omega^2$ and a sequence $(\beta_n)$ of $C^1 \cap W^{1,\infty}$ such that $0 \le \beta_n(s) \le s^2$, $|\beta'_n(s)| \le 2|s|$,  $\beta_n \to s^2$, $\beta_n'(s) \to 2s$, and passing to the limit we deduce 
\bean
&& \int_\OO f_T^2 \widetilde \omega^2  + \int_{\Gamma_{2}^+} J   \gamma f^2 \widetilde \omega^2   + \int_{\Gamma_0}  K   \gamma f^2 \widetilde \omega^2  
+  \int_\UU \fraka (\partial_{y} ( \omega  f ))^2  + \int_\UU f^2 \varpi_-  \widetilde \omega^2
\\
&&\quad \le 
  \int_\UU f^2 \varpi_+  \widetilde \omega^2  +   \int_\OO f_0^2 \widetilde \omega^2  
 +  \int_{\Gamma^+_{2}} J   \frak g^2   \widetilde \omega^2 , \quad \forall \, T > 0.
 \eean
Applying the Gronwall lemma to the function $t \mapsto \| f_t \|_{L^2_{\tilde\omega}}^2$ which   belongs to $L^1(0,T)$ (because $f \in \HH)$ and is lsc (because of the regularity  of $f$ obtained in Step 4), we deduce 
\bear
\label{eq:Gronwall-bdL2}
&&\| f_t \|^2_{L^2_{\tilde\omega}} + \int_0^t \left( \| \gamma f_s \|^2_{L^2_{\tilde\omega}(\Sigma_2^+ \cup \Sigma_0 ; d \xi_1)} +    \| f \|^2_{H^{1,\dagger}_{\tilde\omega}} \right) \, e^{\lambda_0(t-s)} \, d  s 
\\ \nonumber
&&\qquad\le
 \| f_0 \|^2_{L^2_{\tilde\omega}} e^{\lambda_0 t} + \int_0^t  \| \mathfrak{g}_s \|^2_{L^2_{\tilde\omega}(\Sigma_2^-; d \xi_1)}   \, e^{\lambda_0(t-s)} \, d  s,
\eear
for any $t > 0$, with  $\lambda_0 := \sup \varpi_+ < \infty$ and  
$$
\| f \|^2_{H^{1,\dagger}_{\tilde\omega}} :=  \int_\OO a_*  (\partial_{y} ( \omega  f ))^2  + \int_\OO f^2 \varpi_-  \widetilde \omega^2. 
$$

 \smallskip\noindent
\textit{Step 6.} With that last estimate at end, we may pass to the limit in \eqref{eq:renormf-2} written with $\beta_n \to s^2$ and $\psi_n \to \widetilde \omega^2$ and we deduce the identity 
 \bean
&& \| f_{t_2} \|_{L^2_{\tilde\omega}}^2  +  \int_{t_1}^{t_2} \!\! \int_{\Sigma_{2}^+} J   \gamma f^2 \widetilde \omega^2  +  \int_{t_1}^{t_2} \!\! \int_{\Sigma_0} K   \gamma f^2 \widetilde \omega^2  
+ \int_{t_1}^{t_2} \!\!  \int_\OO \fraka (\partial_{y} ( \omega  f ))^2  
\\
&&\quad = \int_{t_1}^{t_2}  \!\!   \int_\OO f^2 \varpi  \widetilde \omega^2  +   \| f_{t_1} \|_{L^2_{\tilde\omega}}^2 
 + \int_{t_1}^{t_2} \!\!  \int_{\Sigma^-_{2}} J   \frak g^2   \widetilde \omega^2 ,  
 \eean
for any $0 \le t_1 \le t_2 < \infty$, in particular $t \mapsto  \| f_{t} \|_{L^2_{\tilde\omega}}^2$ is continuous. Together with the already known weak continuity property, we classically deduce that $f \in C(\R_+;L^2_{\tilde\omega})$. Similarly, taking  $\beta_n \to s_-^2$ and $\psi_n \to \widetilde \omega^2$ in \eqref{eq:renormf-2}, we deduce 
 \bean
  \| f_-(t) \|_{L^2_{\tilde\omega}}^2 
 \le \int_{0}^{t}  \!\!   \int_\OO f_-^2 \varpi  \widetilde \omega^2  +   \| f_{0-} \|_{L^2_{\tilde\omega}}^2 
 + \int_{t_1}^{t_2} \!\!  \int_{\Sigma^+_{2}} J   \frak g_-^2   \widetilde \omega^2, \quad \forall \, t > 0. 
 \eean
 Thanks to the Gronwall lemma, we deduce the positivity property: $f(t) \ge 0$ if $ \frak g \ge 0$ and $f_0 \ge 0$. 
 Similarly, for two solutions $f_1$ and $f_2$ to the inflow problem \eqref{eq:VCklinear},  \eqref{eq:VCktBd1-linear}, \eqref{eq:VCktBd2}, \eqref{eq:LionsBd2bis} in the sense of the variational formulation  \eqref{eq:varTOrenorm-0}, the difference $f := f_2 - f_1$ is also a solution to the  variational problem  \eqref{eq:varTOrenorm-0} but associated to $f_0 = {\frak g} = 0$. Applying the conclusion \eqref{eq:Gronwall-bdL2} to that solution, we  get $f = 0$, and we have thus proved the uniqueness of the solution to the inflow problem \eqref{eq:VCklinear},  \eqref{eq:VCktBd1-linear}, \eqref{eq:VCktBd2}, \eqref{eq:LionsBd2bis}.

\smallskip\noindent
\textit{Step 7.}  
We briefly explain how we may deduce the existence and uniqueness of a solution to the linear  VCk equation \eqref{eq:VCklinear},   \eqref{eq:LinearVCkBd}.
We define $g_0 = 0$ and next recursively the sequence $(g_k)$ defined for $k \ge 1$ as the solution to the equation
\bean
&&\partial_t g_k + \partial_v(J g_k) + \partial_y (K g_k - \fraka \partial_y g_k) = 0  \, \hbox{ in } \, \UU, 
\\
&&g_k(0) = f_0 \, \hbox{ in } \, \OO, \quad  K g_k - \fraka \partial_y g_k = 0 \, \hbox{ on } \, \Gamma_0, 
\\
&&\gamma g_k = 0 \, \hbox{ on } \, \Gamma_1,   \quad J(0,\cdot)  \gamma_- g_k =  J(v_F,\cdot)  \gamma_+  g_{k-1} \, \hbox{ on } \, \Gamma_{2-}
\eean 
provided by the previous steps. The Cesaro means
$$
f_k := {1 \over k} (g_1 + \cdot + g_k)
$$
is then a solution to the equation 
\bean
&&\partial_t f_k + \partial_v(J f_k) + \partial_y (K f_k - \fraka \partial_y f_k) = 0  \, \hbox{ in } \, \UU, 
\\
&&f_k(0) = f_0 \, \hbox{ in } \, \OO, \quad  K f_k - \fraka \partial_y f_k = 0 \, \hbox{ on } \, \Gamma_0, 
\\
&&\gamma f_k = 0 \, \hbox{ on } \, \Gamma_1,   \quad J(0,\cdot)  \gamma_- f_k =  (1-1/k) J(v_F,\cdot)  \gamma_+  f_{k-1} \, \hbox{ on } \, \Gamma_{2-}. 
\eean 
Using repeatingly the  estimates \eqref{eq:Gronwall-bdL2}  corresponding to the equations on $(g_k)$, we deduce that $(g_k)$ is bounded in $\HHH_T$, $\forall \, T >0$. 
Similarly, summing up the estimates \eqref{eq:Gronwall-bdL2} corresponding to the equations on $g_1$, \dots, $g_k$ and using the elementary inequality
$$
\bigl(  {1 \over k} (g_1 + \cdot + g_k) \bigr)^2 \le  {1 \over k} (g^2_1 + \cdot + g^2_k), 
$$
we deduce 
\bear
\label{eq:Gronwall-bdL2-fk}
 \| f_k(t) \|^2_{L^2_{\tilde\omega}} + \int_0^t    \| f_k(s) \|^2_{H^{1,\dagger}_{\tilde\omega}}   \, e^{\lambda_0(t-s)} \, d  s 
\le 
 \| f_0 \|^2_{L^2_{\tilde\omega}} e^{\lambda_0 t}, \quad \forall \, t > 0. 
\eear
We deduce that the exist $f \in \HH_T \cap L^\infty(0,T;L^2_\omega)$, $\forall \, T > 0$, and a subsequence $(f_{k_n})$ such that $f_{n_k} \wto f$ in  $\HHH_T \cap L^\infty(0,T;L^2_\omega)$. Passing first to the limit in the variational formulation of the equation on $(f_{k_n})$ with test functions in $C^1_c(\UU \cup \Gamma_0)$, we deduce 
$$
 \int_\UU   f (- \partial_t  - J \partial_v)  \psi  + \int_\UU (\fraka \partial_y f  - K   )  \partial_y  \psi 
=  0, 
$$
for any $ \psi \in C_c^1(\UU \cup \Gamma_0)$. Repeating the arguments of sections 2, 3 and 4, we know that there exists a trace function $\gamma f$ on $\Gamma$ connected to $f$ through the renormalized formulations \eqref{eq:renormf-2} and \eqref{eq:renormf-3}. 
Proceeding as in section 3, we also get that $(\gamma g_{k_n})$ and  $(\gamma f_{k_n})$ are bounded in $L^2(\Gamma^0_{12}, d\xi_1)$ and in $L^2(\Gamma^{v_F}_{12}, d\xi_2)$.
Passing to the limit  in the variational formulation of the equation on $(f_{k_n})$ with test functions in $C^1_c(\UU \cup \Gamma^0_{12})$ and  in $C^1_c(\UU \cup \Gamma^{v_F}_{12})$, we then deduce first 
$$
\gamma f_{k_n}   \wto \gamma f \ \hbox{ in } \ L^2(\Gamma^0_{12}, d\xi_1) \cap L^2(\Gamma^{v_F}_{12}, d\xi_2) \ \hbox{ as } \ k \to \infty 
$$
and next  
$$
\gamma f_{k_n - 1} = {k_n \over k_n -1} \gamma f_{k_n} - {1 \over k_n -1} \gamma g_{k_n}  \wto \gamma f \ \hbox{ in } \  L^2(\Gamma^{v_F}_{12}, d\xi_2) \ \hbox{ as } \ k \to \infty . 
$$
We may then pass to the limit in the variational formulation associated to the equation on $f_{k_n}$ with test functions in $C_c^1(\UU \cup \Gamma_-)$ and we deduce that $f$ is a variational solution 
 to the linear  VCk equation \eqref{eq:VCklinear},   \eqref{eq:LinearVCkBd},   and more precisely
\beqn\label{eq:var}
 \int_\UU   f (- \partial_t  - J \partial_v)  \psi + \int_\UU (\fraka \partial_y f  - K   )  \partial_y  \psi
= \int_0^T \!\! \int_{y_F}^\infty   (J \gamma f)(t,v_F,y)  \psi(t,0,y), 
+ \int_\OO f_0 \psi(0,\cdot) 
\eeqn
for any $ \psi \in C_c^1(\UU \cup \Gamma_-)$.  Repeating the arguments of Step 6, we similarly establish the further properties of the solution $f$. 
\end{proof}

 \bigskip 
   \section{Ultracontractivity}
\label{sec:ultra}

This section is dedicated to the proof of the ultracontractivity Theorem~\ref{theo-Ultra}, in weighted Lebesgue spaces $L^p_\omega$ for a  strongly confining weight function $\omega := e^{\alpha y}$,   $\alpha > 0$. 

 \subsection{A boundary penalization $L^2$ estimate}
 We take back the $L^2$ estimate established in Section~\ref{subsec:GrowthEstimPrimal} and improve it by introducing a penalization of the neighborhood of the boundary $\Sigma_{12}$ thanks to a suitable power of the distance 
 $\delta(v) :=   \min (v,v_F-v)$.

\begin{lem}\label{lem:estimL2+}
Assume $\omega := e^{\alpha y}$, $\alpha > 0$. 
There exist a weight function $\widecheck\omega = \widecheck \omega(v,y) \simeq \omega$ and a constant $\widecheck\kappa \ge 0$  such that  
any solution $f $ to the linear VCk equation  \eqref{eq:VCklinear}-\eqref{eq:LinearVCkBd}  
satisfies 
 \beqn\label{eq:lem:estimL2+}
 \int_0^T \varphi^2 \int_\OO  \Bigl\{ f^2 \omega^2 \Bigl[ {J^2 \over \delta^{1/2} \langle y \rangle^2} +  \alpha \langle y \rangle \Bigr] + (\partial_y ( f \widecheck\omega) )^2 \Bigr\}  \le \widecheck\kappa 
\int_0^T ( \varphi \varphi'_+ + \varphi^2)  \int_\OO  f^2  \omega^2,
\eeqn 
 for any $\varphi \in C^1_c(0,T)$ and where   $\widecheck\kappa$ only depends on $a^*,a_*,y_*$. 
 \end{lem}

\begin{proof}[Proof of Lemma~\ref{lem:estimL2+}] We come back to the proof of Lemma~\ref{lem:GrowthLp}. We define
\beqn\label{eq:def_baromega}
 \widecheck\omega^2 := \widetilde \omega^2 \frak P, \quad \frak P := 1 - \frac12 {\delta^{1/2}  \over \delta_*^{1/2}} {\delta' J \over \langle J \rangle^2}, 
\eeqn
with $\delta_* := \sup [\delta (\delta')^2]$,   $\widetilde \omega^2$ defining during the  proof of Lemma~\ref{lem:GrowthLp} and where we abuse notation by rather denoting $\delta \in C^2(\R)$  such that $\delta  > 0$  on $(0,v_F)$ and   $\delta(v)=   \min (v,v_F-v)$ for any $v \in (0,v_F/3) \cup (2v_F/3,v_F)$ so that it is equivalent to establish \eqref{eq:lem:estimL2+} for the previously defined  $\delta$ or this smooth variant.
  Proceeding exactly as during the  proof of Lemma~\ref{lem:GrowthLp}, we have 
\bean
 \frac12  {d \over dt} \int_\OO f^2 \widecheck\omega^2
&=& -   \int_\OO \fraka (\partial_y (f\widecheck\omega) )^2 +  \int_\OO  f^2 \widecheck\omega^2 \varpi
\\
&& + \int_{\Sigma_0} \frac12 K    n_0     (f \widecheck\omega)^2  
-   \frac1p \int_{\Sigma_2} J f^2 \widetilde\omega^2 n_v,
\eean
with 
\bean
\varpi :=   \fraka \bigl( {\partial_y \check\omega \over \check\omega} \bigr)^2
+ K  {\partial_y \check\omega \over \check\omega} +  J  {\partial_v \check\omega \over  \check\omega} 
- \frac12 \partial_y K 
- \frac12 \partial_v J, 
\eean
and thus again 
\bean
 \frac12  {d \over dt} \int_\OO f^2 \widecheck\omega^2
 +   \int_\OO a_* (\partial_y (f\widecheck\omega) )^2  \le \int_\OO  f^2 \widecheck\omega^2 \widecheck \varpi. 
\eean
We observe that 
$$
{\partial_z \check \omega \over \check \omega} = {\partial_z \widetilde \omega \over  \widetilde  \omega} + \frac12 {\partial_z \frakP  \over \frakP}  
$$
for $z=y,v$, and thus 
$$
\widecheck \varpi = \widetilde \varpi +  \fraka {\partial_y\widetilde \omega \over  \widetilde  \omega}   {\partial_y \frakP  \over \frakP}  + \frac{K}2  {\partial_y \frakP  \over \frakP}  
 + \frac{J}2  {\partial_v \frakP  \over \frakP}  
$$
with $\widetilde \varpi $ defined by \eqref{def:varpi}, 
$$
\partial_v \frakP =  - \frac14 {1 \over \delta^{1/2}} {(\delta')^2 \over \delta_*^{1/2}} {  J \over \langle J \rangle^2} + {\phi_1 \over \langle J \rangle},
\quad  \phi_1 \in L^\infty(\OO), 
$$
and
$$
\partial_y \frakP =  - \frac12 {\delta^{1/2}  \over \delta_*^{1/2}}  \delta' \partial_y {J \over \langle J \rangle^2} =  {\phi_2 \over \langle J \rangle^2},
\quad  \phi_2 \in L^\infty(\OO).
$$
We deduce that 
\beqn
\label{eq:checkvarpi}
\widecheck \varpi \le \check\kappa -  {1 \over 8 \delta_*^{1/2}} {1 \over \delta^{1/2}} {J^2 \over \langle y \rangle^2}   - \alpha    \langle y \rangle
\eeqn
with $\check\kappa \le C (1 + \| \fraka \|_{L^\infty} +  \| \frakb \|_{L^\infty})$ and where we observe that $\langle J \rangle \simeq \langle y \rangle$. 
The previous estimate gives
\bean
 \frac12  {d \over dt} \int_\OO f^2 \widecheck\omega^2
 +   \int_\OO f^2  \widecheck\omega^2 \Bigl( {1 \over 8 \delta_*^{1/2}} {1 \over \delta^{1/2}} {J^2 \over \langle y \rangle^2}  + \alpha    { \langle y \rangle} \Bigr)
  +   \int_\OO a_* (\partial_y (f\widecheck\omega) )^2   \le \check\kappa \int_\OO  f^2 \omega^2. 
\eean
We conclude by multiplying the above equation by $\varphi^2$ and integrating in the time variable. 
\end{proof}
 
\medskip
 
We reformulate that last result in a more tractable one by using several times the following estimate 
\beqn\label{eq:HardyIneq}
  \int_a^c \frac1{|y-b|^{\mu}} h^2  dy \lesssim  \int_a^c ((\partial_y h)^2 + h^2) dy, 
\eeqn
for   $\mu \in (0,1)$ and any $h \in H^1(a,c)$, $a,b \in \R$, $c \in \R  \cup \{+\infty\}$, $a \le b \le c $. The estimate \eqref{eq:HardyIneq} is an immediate consequence of the classical embedding $H^1(a,c) \subset L^\infty(a,c)$ used in the region where $|y-b|  \le  1$ in the LHS integral of \eqref{eq:HardyIneq}. 

\begin{lem}\label{cor1:L2estim+}
Assume $\omega := e^{\alpha y}$,   $\alpha > 0$.
With the notations of Lemma~\ref{lem:estimL2+},  any solution $f $ to the linear VCk equation  \eqref{eq:VCklinear}-\eqref{eq:LinearVCkBd}  satisfies 
$$
\int_0^T \varphi^2 \int_\OO  f^2 {1  \over \delta^{1/9}} {\omega^2 \over \langle y \rangle^2} \lesssim
\int_0^T (\varphi \varphi'_+  + \varphi^2) \int_\OO  f^2  \omega^2
$$
for any $\varphi \in C^1_c(0,T)$ 
\end{lem}

\begin{proof}[Proof of Lemma~\ref{cor1:L2estim+}]
We define $\OO = \OOO_1 \cup \OOO_2 \cup \OOO_3 \cup \OOO_4$ with 
\bean
\OOO_1 &:=& \{ (v,y); \, 0 < v  < \tfrac12 v_F,  \  \tfrac12 v_E \le v + y_L v^{29/36}, \ y > v^{7/36} \} , 
\\
\OOO_2 &:=& \{ (v,y); \, 0 < v  < \tfrac12 v_F, \  \tfrac12 v_E > v + y_L v^{29/36}, \ y > v^{7/36} \}  , 
\\
\OOO_3 &:=&  \{ (v,y); \, 0 < v  < \tfrac12 v_F,  \ y \le v^{7/36} \}  , 
\\
\OOO_4 &:=& \{ (v,y); \,  \tfrac12 v_F \le  v  < v_F, \ 0 \le y <  \tfrac12 y_F\}  , 
\\
\OOO_5 &:=& \{ (v,y); \,  \tfrac12 v_F \le  v  < v_F, \  y >  \tfrac12 y_F\}  , 
\eean
and we estimate separately each of the term
$$
\int_0^T\int_{\OOO_i} \varphi^2 f^2 {\omega_0^2 \over \delta^{1/9}}, 
$$
for $i=1, \dots , 5$ and where we use the shorthand $\omega_0 := \omega/\langle y \rangle$. 

\smallskip\noindent
$\bullet$ On $\OOO_1$, we have $\delta \ge \delta_1 > 0$, so that 
$$
\int_0^T\int_{\OOO_1} \varphi^2 f^2 {\omega_0^2 \over \delta^{1/9}} \lesssim \int_0^T\int_{\OOO_1} \varphi^2 f^2 \omega_0^2.
$$

\smallskip\noindent
$\bullet$ On $\OOO_2$, we observe that 
\bean
J &\ge& y(v_E - v - y_L v^{29/36}) \ge \tfrac12v_E y, 
\eean
so that 
$$
{J^2 \over \delta^{1/2}} \gtrsim {y^2 \over v^{1/2}} \gtrsim {1\over v^{1/9}} %+ {1\over y^{2}} 
\gtrsim {1\over \delta^{1/9}},
$$
where we have used $y >  v^{7/36}$ and next $v = \delta $ in the two last inequalities. 
We deduce 
$$
\int_0^T\int_{\OOO_2} \varphi^2 f^2 {\omega_0^2 \over \delta^{1/9}} \lesssim \int_0^T\int_{\OOO_2} \varphi^2 f^2 \omega_0^2 {J^2 \over \delta^{1/2}}. 
$$

\smallskip\noindent
$\bullet$ On $\OOO_3$, 
we have 
\bean
\int_0^T\int_{\OOO_3} \varphi^2 f^2 {\omega_0^2 \over \delta^{1/9}} 
&\lesssim& \int_0^T \!\! \int_{0}^{v_F/2} \!\! \int_0^{v_F^{7/36}}  \varphi^2 (\omega f) ^2  {1 \over y^{4/7}} 
\\
&\lesssim& \int_0^T \!\! \int_{0}^{v_F/2} \!\! \int_0^{\infty}  \varphi^2   (\partial_y  (\omega f) )^2 
\eean
where we have used that $y^{4/7} \le \delta^{1/9}$ on the first line and  the inequality \eqref{eq:HardyIneq}  in the last line. 

\smallskip\noindent
$\bullet$
On $\OOO_4$, we observe  that  there exists $v^*_F \in [\tfrac12 v_F,v_F)$ and $J_* > 0$ such that 
\bean
J^2 \ge  J^2_* > 0 \hbox{ if } (v,y) \in (v^*_F,v_F) \times (0,y_F/2), 
\eean
so that 
\bean
\frac1{\delta^{1/9}}
\le \frac1{\delta(v^*_F)^{1/9}} {\bf 1}_{v < v^*_F} + {J^2 \over J^2_*} \frac1{\delta^{1/9}} {\bf 1}_{v > v^*_F}  
  \lesssim 1 +  {J^2 \over  \delta^{1/2}}  
\eean
on $\OOO_4$. 
We deduce 
\bean
\int_0^T\int_{\OOO_4} \varphi^2 f^2 {\omega_0^2 \over \delta^{1/9}} 
&\lesssim& \int_0^T \!\! \int_{\OOO_4}  \varphi^2 f^2 \omega_0^2  \Bigl( 1+  \frac{J^2}{\delta^{1/2}} \Bigr).
\eean

\smallskip\noindent
$\bullet$ On $\OOO_5$, we first observe that 
$$
  {1 \over \delta^{1/9}} \le {J^2 \over \delta^{1/2}} + {1 \over |J|^{4/7}} , 
$$
thanks to the Young inequality, 
 and that 
$$
\int_{y_F/2}^\infty  {f^2 \over |J|^{4/7}} \omega_0^2 dy \lesssim \int_0^\infty  (\partial_y (f\omega))^2 dy 
$$
from the inequality \eqref{eq:HardyIneq} again and the fact that we may write $|J(v,y)| = \zeta |y-c|$, with $\zeta = \zeta(v) \in [\zeta_*,\zeta^*]$,  $c = c(v) \in [c_*,c^*]$, $\zeta_*,\zeta^*,c_*,c^* \in (0,\infty)$. 
We deduce 
\bean
\int_0^T\int_{\OOO_5} \varphi^2 f^2 {\omega_0^2 \over \delta^{1/9}} 
&\lesssim& \int_0^T \!\! \int_{v_F/2}^{v_F}  \int_0^{\infty}  \varphi^2   \Bigl(   (\partial_y (f\omega))^2 + f^2 \omega_0^2 \frac{J^2}{\delta^{1/2}} \Bigr).
\eean
We conclude by gathering the above five contributions and by using the previous estimate \eqref{eq:lem:estimL2+}. 
\end{proof} 

Interpolating the two previous estimates, we conclude with the following formulation of penalization of the boundary (in the $v$ variable) and the infinity (in the $y$ variable). 

\begin{prop}\label{prop:estimL2++}
Assume $\omega := e^{\alpha \langle y \rangle}$, with $\alpha > 0$.
Any   solution $f $ to the linear VCk equation  \eqref{eq:VCklinear}-\eqref{eq:LinearVCkBd}  
 satisfies 
$$
\int_0^T \varphi^2 \int_\OO  f^2 \omega^2 \Bigl[  {1  \over \delta^{1/27}}+ \langle y \rangle  \Bigr]  \lesssim
\int_0^T (\varphi \varphi'_+  + \varphi^2) \int_\OO  f^2 \omega^2
$$
for any $\varphi \in C^1_c(0,T)$.  
\end{prop}

\begin{proof}[Proof of Proposition~\ref{prop:estimL2++}] The estimate immediately follows from  the Young inequality
$$
{1  \over \delta^{1/27}} = {1  \over \delta^{1/27} \langle y \rangle^{2/3} } \langle y \rangle^{2/3} \le {1  \over \delta^{1/9} \langle y \rangle^{2} } +  \langle y \rangle , 
$$
together with  Lemma~\ref{lem:estimL2+} and Lemma~\ref{cor1:L2estim+}. 
\end{proof}

 \subsection{Gain of integrability estimate}
 \label{subseq:GainL2}
 We introduce the function 
\beqn
\label{eq:def-barf}
 \bar f :=  f  \psi, \quad \psi := \varphi(t) \chi (v)   \omega_0(y) , 
\eeqn
 with $ \varphi \in C^1_c((0,T))$, $\chi \in C^1_c((0,v_F))$, $0 \le \varphi, \chi \le 1$,  $\omega_0 :=   e^{\alpha y}$ for $y \ge y_0 > 0$, $\alpha > 0$, $0 \le \omega_0 \in C^1_c((0,\infty))$, which is a solution to 
\beqn\label{eq:KFP-interior-barf}
 \partial_t \bar f + \hat J \partial_v \bar f+   \check K \partial_y \bar f - a \partial_{yy}^2 \bar f + \alpha \langle y \rangle \bar f =   F
\eeqn
 on $\R_+ \times \R^2$, with  
\bean
&&\hat J := J(v,y), \ \hbox{if} \ y > 0, \quad
\hat J  :=  J(v,-y), \ \hbox{if}  \ y < 0, 
\\
&&\check K := K(v,y), \ \hbox{if} \ y > 0, \quad
\check K :=  - K(v,-y), \ \hbox{if}  \ y < 0, 
\\
&& F :=  (\partial_t \varphi) \chi \omega_0  f +  \hat J \varphi ( \partial_v \chi) \omega_0 f  - (\partial_v  \hat J)  \varphi   \chi \omega_0 f 
 +  \check  K \varphi \chi ( \partial_y \omega_0)  f  - (\partial_y \check  K)  \varphi   \chi \omega_0 f 
 \\&&
 \qquad
- a    \varphi    \chi  (\partial_{yy}^2  \omega_0) f 
- 2 a   \varphi    \chi  (\partial_y  \omega_0)   (\partial_y  f)
+ \alpha \langle y \rangle   \varphi   \chi \omega_0 f 
\eean
  We first focus on the homogeneous equation 
\beqn\label{eq:KFP-interior-g}
 \partial_t g+ \bar J \partial_v g   + \bar K  \partial_y g  - \fraka \partial_{yy}^2 g  + \alpha \langle y \rangle g = 0,
  \quad \hbox{on}\quad \R_+ \times \R^2,
\eeqn
complemented with the initial condition 
 $$
 g(0) = g_0, \quad \hbox{on}\quad \R^2, 
 $$
 where 
 $$
 \bar K, \bar J \in C^1_c((0,v_F) \times \R), \quad \bar K = \check K, \ \bar J = \hat J \ \hbox{ on supp} \,  \chi \times \R. 
 $$
 More precisely, we take $ \bar K =  \bar \chi \check K$, $ \bar J =  \bar \chi \hat J$, $ \bar \chi \in C^1_c((0,v_F))$, ${\bf 1}_{\hbox{\rm supp} \chi} \le \bar\chi \le 1$, 
 and $\alpha > \| \bar\chi' \| (v_E+y_Lv_F+1)$. 
We establish the following series of growth estimates and gain of integrability result.

\begin{lem}\label{lem:LpInterior}  
For any $p \in [1,\infty]$, any solution $g$ to the homogeneous equation  \eqref{eq:KFP-interior-g} satisfies 
\beqn\label{eq:propLpInteriorg}
\| g_t \|_{L^p} \le  \|g_0 \|_{L^p}, \quad \forall \, t \ge 0.
\eeqn
\end{lem}

\smallskip\noindent
{\sl Proof of Lemma~\ref{lem:LpInterior}.} For $p \in  (1,\infty)$, any nonnegative solution $g$ to \eqref{eq:KFP-interior-g} satisfies 
\bean
\frac1p\frac{d}{dt} \int_{\R^2} g^p 
=  \frac1p  \int_{\R^2} g^p ( \partial_v \bar J + \partial_y \bar K - p \alpha \langle y \rangle) - (p-1) a \int_{\R^2} (\partial_y g)^2 g^{p-2}
\le 0, 
\eean
because $\partial_v \bar J + \partial_y \bar K \le \bar\chi' \check J \le \alpha \langle y \rangle$. 
We conclude to \eqref{eq:propLpInteriorg} thanks to the Gronwall lemma when $p \in (1,\infty)$ and by passing to the limit $p \to 1$ and $p \to \infty$ in  \eqref{eq:propLpInteriorg} for the two limit cases.  
\qed

\begin{prop}\label{prop:L2H1Interior} 
Any solution $g$ to the homogeneous equation  \eqref{eq:KFP-interior-g} satisfies  
\beqn\label{eq:propL2H1Interiorg}
\| g_t \|_{H^1_vL^2_y \cap L^2_vH^2_y} \lesssim  {1 \over t^2}  \|g_0 \|_{L^2}, \quad \forall \, t \in (0,1). 
\eeqn

\end{prop} 

\smallskip\noindent
{\sl Proof of Proposition~\ref{prop:L2H1Interior}.} We use a regularization argument in the spirit of  
H{\'e}rau~\cite{MR2294477} (see also Villani~\cite[Appendix]{MR2562709}). More precisely, we define 
$$
\FF(t,g) :=    \eps^{-3}
 \| g \|^2_{L^2} + \eps^4 t^4 \|\partial_v g \|_{L^2}^2 + \eps^3 t^3(\partial_v g,\partial_y g)_{L^2}  + \eps  t^2  \|\partial_y g \|_{L^2}^2   + 
\eps^4 t^4 \|\partial^2_{yy} g \|_{L^2}^2, 
$$
for $t,\eps \in (0,1)$, 
and for a  solution $g$ to the homogeneous equation  \eqref{eq:KFP-interior-g}, we estimate the derivative quantity $\tfrac{d}{dt} \FF(t,g_t)$. 
For that purpose, we start computing each square of norm or scalar product on $g$ involved separately.

\smallskip\noindent
{\sl Step 1. } From Lemma~\ref{lem:LpInterior}, we already know that 
\bean
\frac12 {d \over dt} \| g \|^2_{L^2} \le   -  \| g  \langle y \rangle^{1/2} \|^2_{L^2}  - a \| \partial_y g \|^2_{L^2}.
\eean
Similarly, we compute 
\bean
  \frac12{d \over dt}   \int_{\R^2} (\partial_v g)^2
=   \int_{\R^2} (\partial_v g)^2 \Bigl( \frac12\partial_y \bar K -  \frac12 \partial_v \bar J -  \alpha \langle y \rangle \Bigr) 
  -  a \int_{\R^2} (\partial_{vy}^2 g)^2 , 
\eean
so that 
\bean
\frac12 {d \over dt} \| \partial_v g \|^2_{L^2} \le C_1   \|  \partial_v g  \|^2_{L^2} -   \| \langle y \rangle^{1/2} \partial_v g  \|^2_{L^2}  - a \| \partial_{vy}^2 g \|^2_{L^2}.
\eean
We next compute 
\bean
 {d \over dt}   \int_{\R^2} \partial_v g \partial_y g 
&=&   \int_{\R^2} \Bigl(  - \frac12 \partial_v \bar J (\partial_y g)^2  - {\frac 12} \partial_y \bar J (\partial_v g)^2   - 2  a  \partial_{vy}^2 g \partial_{yy}^2 g - 2 \alpha \langle y \rangle   \partial_v g  \partial_y g  \Bigr),  
\eean
so that 
\bean
{d \over dt}   \int_{\R^2} \partial_v g \partial_y g 
&\le&  - C_{21} \|  \partial_v g \|^2_{L^2} +  C_{22} \| \langle y \rangle^{1/2} \partial_y g \|^2_{L^2} 
\\
&&+
C_{23}   \|  \partial_{vy} g  \|_{L^2}\|  \partial_{yy} g  \|_{L^2} + C_{24}   \|   \langle y \rangle^{1/2} \partial_v g  \|_{L^2}\|   \langle y \rangle^{1/2} \partial_y g  \|_{L^2} .
\eean
We also compute 
\bean
\frac12{d \over dt}   \int_{\R^2} (\partial_y g)^2
&=&  \int_{\R^2} (\partial_y g)^2 \Bigl( \frac12 \partial_v \bar J - \frac12\partial_y \bar K -   \alpha \langle y \rangle \Bigr)  
\\&&+  \int_{\R^2} g^2 \frac\alpha2 \partial^2_{yy} \langle y\rangle
-   \int_{\R^2} \partial_{y} \bar J \partial_{y} g \partial_{v} g -  a \int_{\R^2} (\partial_{yy}^2 g)^2 , 
\eean
so that 
\bean
\frac12 {d \over dt} \| \partial_y g \|^2_{L^2} &\le& -   \|  \langle y \rangle^{1/2} \partial_y g  \|^2_{L^2} - a \| \partial_{yy}^2 g \|^2_{L^2} 
\\
&&+ C_{31}   \|  \partial_y g  \|^2_{L^2}  +
C_{32}   \|   g  \|^2_{L^2}  + C_{33}   \|    \partial_v g  \|_{L^2}\|   \partial_y g  \|_{L^2}  .
\eean
We finally  compute 
\bean
  \frac12{d \over dt}   \int_{\R^2} (\partial_{yy} g)^2
&=&   
 \int_{\R^2} (\partial_{yy} g)^2 \Bigl( \frac12 \partial_v \bar J - \frac32\partial_y \bar K -   \alpha \langle y \rangle \Bigr) 
 +  \int_{\R^2} (\partial_{y} g)^2   2 \alpha  \partial^2_{yy} \langle y \rangle
 \\
 && -  \int_{\R^2} g^2  \frac \alpha 2  \partial^4_{yyyy} \langle y \rangle
  -  a \int_{\R^2} (\partial_{yyy}^2 g)^2 -  2 \int_{\R^2} \partial_y \bar J  \partial_{vy} g  \partial_{yy} g , 
\eean
so that 
\bean
\frac12 {d \over dt} \| \partial_{yy} g \|^2_{L^2} 
&\le&  
+ C_{41}   \|  \partial_{yy}^2 g \|^2_{L^2} 
 + C_{42}   \|  \partial_{y} g \|^2_{L^2} + C_{43}   \|   g \|^2_{L^2}
+   C_{44}   \|    \partial_{vy} g   \|_{L^2}\|   \partial_{yy} g  \|_{L^2}  .
\eean

\smallskip\noindent
{\sl Step 2. } 
The derivative of $\FF$ takes  the form
\bean
    {d \over dt}\FF(t,g_t) &=&
  \eps^{-3} 
     {d \over dt} \| g \|^2_{L^2}
+4  \eps^4t^{3} \| \partial_v g \|^2_{L^2}+ \eps^4t^{4}{d \over dt} \| \partial_v g \|^2_{L^2}
    \\
    &&+3 \eps^3t^2  (\partial_v g,\partial_y g)_{L^2} +\eps^3 t^{3}{d \over dt}   (\partial_v g,\partial_y g)_{L^2}
     \\
    &&+ 2 \eps  t  \| \partial_y g \|^2_{L^2} +\eps t^{2}{d \over dt} \| \partial_y g \|^2_{L^2}
  \\
    &&+ 4 \eps^4 t^{3} \| \partial_{yy} g \|^2_{L^2} +\eps^4t^{4}{d \over dt} \| \partial_{yy} g \|^2_{L^2}, 
\eean
for any $\eps, t \in (0,1)$. 
Using the estimates established in the first step and the Young inequalities
\bean
&&( \eps^3 t^3  C_{23}   + 
2 \eps^4t^{4}  C_{44} )   \|     \partial_{vy} g   \|_{L^2}\|  \partial_{yy} g  \|_{L^2}
\\
&&\qquad\le 
( \eps^{3/2} t^2  C_{23}   +  \eps^3t^{4}  C_{44} )  \|  \partial_{yy} g  \|_{L^2}^2 
+ ( \eps^{9/2} t^4  C_{23}   +  \eps^5t^{4}  C_{44} )   \|     \partial_{vy} g   \|_{L^2}^2, 
\\
 &&  \bigl(    2 \eps t^{2}  C_{33}  + 3 \eps^3t^2  \bigr) \| \partial_v g  \|_{L^2} \|  \partial_y g  \|_{L^2} 
 \\
 &&\qquad \le  \bigl(   \eps^{4}t^{4}  C_{33}  + 2 \eps^4t^4  \bigr) \| \partial_v g  \|_{L^2}^2
 +  \bigl(   \eps^{-2}   C_{33}  +  2 \eps^2  \bigr) \|  \partial_y g  \|_{L^2}^2
\\
 && \eps^3 t^{3}     C_{24}   \|   \langle y \rangle^{1/2} \partial_v g  \|_{L^2}\|   \langle y \rangle^{1/2} \partial_y g  \|_{L^2}
 \\
 &&\qquad \le \eps^{3/2} t^{2}     C_{24}   \|   \langle y \rangle^{1/2} \partial_y g  \|_{L^2}^2+ \eps^{9/2} t^{4}     C_{24}  \|   \langle y \rangle^{1/2} \partial_y g  \|_{L^2}^2, 
 \eean
we immediately deduce 
\bean
{d \over dt}\FF(t,g_t) &\le&
 \| g  \langle y \rangle^{1/2} \|^2_{L^2} (-2\eps^{-3} + 2 \eps C_{32} t^2 + 2 \eps^4 C_{43} t^4 ) 
 \\
 &&
 +  \| \partial_y g \|^2_{L^2}  (-2a \eps^{-3} + 2 \eps C_{31} t^2 + 2 \eps^4 C_{42} t^4 + 2 \eps  t +  \eps^{-2}   C_{33}  +  2 \eps^2)
\\
&&
+  \| \langle y \rangle^{1/2} \partial_y g \|^2_{L^2} (  \eps^{3/2} t^{2}     C_{24}  +  \eps^3 t^3 C_{22} - 2 \eps t^2 ) 
\\
&&
+  \|  \partial_v g \|^2_{L^2} ( 4  \eps^4t^{3} +  2 \eps^4t^{4}   C_1 +  \eps^{4}t^{4}  C_{33}  + 2 \eps^4t^4 - C_{21} \eps^3 t^{3} ) 
\\
&&
  +  \|  \partial_{yy} g \|^2_{L^2} ( -2 \eps a t^{2}  + 4 \eps^4 t^{3}+ 2 \eps^4t^{4}C_{41} + \eps^{3/2} t^2  C_{23}   +  \eps^3t^{4}  C_{44} ) 
\\
&&
+ ( \eps^{9/2} t^4  C_{23}   +  \eps^5t^{4}  C_{44}  - 2 \eps^4 t^4 a )  \|  \partial_{vy} g \|^2_{L^2}
\\
&&
+ (\eps^{9/2} t^{4}     C_{24}  -      2 \eps^4t^{4} )  \| \langle y \rangle^{1/2} \partial_v g  \|^2_{L^2} . 
\eean
Choosing $\eps \in (0,1)$ small enough, all the terms are nonnegative separately, and we thus obtain 
\bean
\FF(t,g_t) \le \FF(0,g_0) = \eps^{-3} \| g_0 \|_{L^2}^2 , \quad \forall t \in  (0,1). 
\eean
Observing that 
$$
\FF(t,g) \ge \tfrac12 \eps^4 t^{4} \bigl (  \| g \|^2_{L^2} +  \|\partial_v g \|_{L^2}^2 +    \|\partial_y g \|_{L^2}^2   +  \|\partial^2_{yy} g \|_{L^2}^2) , 
$$
we then conclude to \eqref{eq:propL2H1Interiorg} with $C_1 := (2\eps^{-7})^{1/2}$. 
\qed

\begin{cor}\label{cor:L2LinftyInterior}
 Any solution $g$ to the homogeneous equation  \eqref{eq:KFP-interior-g} satisfies 
$$
\| g_t \|_{L^\infty} \lesssim {1  \over t^{2}}  \|g_0 \|_{L^2}, \quad \forall \, t \in (0,1).
$$
\end{cor} 

\begin{proof}[Proof of Corollary~\ref{cor:L2LinftyInterior}]
We first observe that the following Sobolev type embedding 
$$
H^1_vL^2_y \cap L^2_vH^2_y  \subset L^\infty(\R^2)
$$
 holds true. 
We may indeed write
$$
\| f \|_{L^\infty} \lesssim \| \hat f \|_{L^1} \lesssim A B
$$
using the Cauchy-Schwarz inequality in the second inequality, with 
\bean
A^2 &:=&  \int_{\R^2} (1+|\xi|^2 +|\eta|^4) |\hat f |^2 d\xi d\eta
\\
 &=&  \int_{\R^2} ( |\hat f |^2 + |\widehat{\partial_v f}|^2  +  |\widehat{\partial^2_{yy} f}|^2 ) d\xi d\eta \le \| f \|_{H^1_vL^2_y \cap L^2_vH^2_y}^2, 
\eean
where we have used the Plancherel identity in the last inequality,
and
\bean
B^2 &:=&  \int_{\R^2} { 1 \over 1 +|\xi|^2 +|\eta|^4}    d\xi d\eta
=   \int_{\R} {  d\eta \over (1   +|\eta|^4)^{1/2}}    \int_{\R} { dz \over 1 + z^2}  < \infty, 
\eean
by performing the change of variables $z := \xi/(1   +|\eta|^4)^{1/2}$. 
We then conclude the proof thanks to  \eqref{eq:propL2H1Interiorg}. \end{proof}

\smallskip
We come back to the analysis of   the linear VCk equation  \eqref{eq:VCklinear}-\eqref{eq:LinearVCkBd}.

\begin{prop}\label{prop:L2LpInterior-barf} 
Assume $\omega := e^{\alpha y}$, $\alpha > 0$.
For any solution $f $ to the linear VCk equation  \eqref{eq:VCklinear}-\eqref{eq:LinearVCkBd},  
the  function $\bar f$ defined by \eqref{eq:def-barf}  satisfies 
\beqn\label{eq-prop:L2LpInterior-barf}
\int_0^T \| \bar f \|^2_{L^{5/2}}  dt \lesssim \|  \chi \|_{W^{1,\infty}}^2 \int_0^T \|   f    \omega\|^2_{L^2} (\varphi^2 +  (\varphi')^2)dt , 
\eeqn
for any $T \in (0,1)$, $\varphi \in C^1_c(0,T)$ and $\chi \in C^1_c(0,v_F)$. 
\end{prop}

\smallskip\noindent
{\sl Proof of Proposition~\ref{prop:L2LpInterior-barf}.} 
We denote by $S^0_t$ the semigroup associated to homogeneous equation \eqref{eq:KFP-interior-g} which is well defined by arguing as in Theorem~\ref{theo:existL2} taking advantage of the estimates established in Lemma~\ref{lem:LpInterior}. 
From Lemma~\ref{lem:LpInterior} and Corollary~\ref{cor:L2LinftyInterior}, we have  
$$
\| S_t^0 \|_{L^2 \to L^2} \lesssim 1, \quad \| S_t^0 \|_{L^2 \to L^\infty} \lesssim {1 \over t^{2}}, 
$$
for any $t \in (0,1)$. By interpolation (Holder inequality), we deduce 
$$
\| S_t^0 \|_{L^2 \to L^{5/2}} \lesssim   {1 \over t^{2/5}}, 
$$
for any $t \in (0,1)$. Because $\bar f$ satisfies \eqref{eq:KFP-interior-barf}, we may use the Duhamel formula together with the fact that $\bar f(0,\cdot) = 0$ on $\R^2$, and we get
$$
\bar f_t = \int_0^t S^0_{t-s} F_s ds.
$$
As a consequence, we have 
\bean
\| \bar f_t \|_{L^{5/2}(\R^2)} 
&\le& \int_0^t \| S_{t-s}^0 \|_{L^2 \to L^{5/2}} \| F_s \|_{L^2} ds 
\\
&\lesssim& \int_0^t { 1 \over (t-s)^{2/5}}  \| F_s \|_{L^2} ds 
\\
&\lesssim&  \Bigl(  \int_0^t {ds  \over (t-s)^{4/5}}\Bigr)^{1/2} \Bigl(  \int_0^t   \| F_s \|_{L^2}^2 ds \Bigr)^{1/2} 
\\
&\le&  t^{1/10} \| F \|_{L^2(\UU)},  
\eean
for any $t \in (0,1)$, where we have used the previous estimate in the second line, the Cauchy-Schwarz inequality in the third line and the very definition of $F$ in the last line. 
We deduce  \eqref{eq-prop:L2LpInterior-barf} by observing that 
$$
|F| \lesssim |\partial_t \varphi f \omega| +  |\varphi \partial_v \chi f  \omega| + |\varphi  f  \omega| +  |\varphi  \partial_y(f \widetilde \omega)|, 
$$
 with $\widetilde \omega$ defined by \eqref{eq:def-widetildeomega}, next by  
integrating the above estimate in the time variable and by using Lemma~\ref{lem:estimL2+}. \qed 

\medskip

We then come back on the analysis of a solution $f $ to the linear VCk equation  \eqref{eq:VCklinear}-\eqref{eq:LinearVCkBd}.  

\begin{prop}\label{prop:L2LpInterior-barf} For any  $\beta> 5/2$,  there exists $C$ such that 
$$
\int_0^T \|  f  \delta^{2\beta/5} {\omega \over \langle y \rangle}\|^2_{L^{5/2}} \varphi^2 dt \le C   \int_0^T \|   f \|_{L^2}^2 ( \varphi^2 + (\varphi')^2)  dt 
$$
\end{prop}

\smallskip\noindent
{\sl Proof of Proposition~\ref{prop:L2LpInterior-barf}.} We take $\chi = \chi_k $ with $\| \nabla \chi_k \|  2^{-k}  \lesssim 1$ uniformly in $k$, and we repeat the proof of Step 2 in \cite[Proposition~3.5]{carrapatoso2024kinetic}. 
\qed

\medskip
We are now able establish the cornerstone estimate of gain of integrability in $(v,y)$ variables for a solution $f $ to the linear VCk equation  \eqref{eq:VCklinear}-\eqref{eq:LinearVCkBd} 
by gathering the last result and Proposition~\ref{prop:estimL2++}.  

\begin{prop}\label{prop:L2Lr} There exists $r > 2$ such that 
$$
\int_0^T  \varphi^2  \|  f \omega \|_{L^r(\OO)}^2 dt \lesssim  \| (\varphi + |\varphi'|)   f \omega \|_{L^2(\UU)}^2, 
$$
for any $\varphi \in C^1(0,T) \cap C_0(0,T)$, $T \in (0,1)$. 
\end{prop}

\begin{proof}[Proof of Proposition~\ref{prop:L2Lr}]
We set $\Delta :=  \delta^{2\beta/5} / \langle y \rangle$, so that  Proposition~\ref{prop:L2LpInterior-barf} writes 
\beqn\label{eq:EstimL2LpBIS}
\int_0^T \|  f   \omega  \Delta \|^2_{L^{5/2}(\OO)} \varphi^2 dt \lesssim   \int_0^T \| f\omega  \|_{L^2(\OO)}^2 ( \varphi^2 + (\varphi')^2)  dt. 
\eeqn
From Proposition~\ref{prop:estimL2++} and the Holder inequality, we have next
\bean
\| f \varphi \omega {\langle y \rangle^{\theta/2} \over \delta^{  (1-\theta)/12} } \|_{L^2} 
&\le& \| f \varphi \omega \langle y \rangle^{1/2} \|_{L^2}^\theta  \| f \varphi \omega {1  \over \delta^{1/12} } \|_{L^2}^{1-\theta}
\\
&\lesssim& \|  (\varphi + |\varphi'| ) f \omega   \|_{L^2}, 
\eean
for any $\theta \in (0,1)$. Choosing $\theta = \theta_0$ such that 
$$
  (1/\theta_0-1) /  6 = 2 \beta/5
$$
and setting $\mu := \theta_0/2$, we thus deduce 
\beqn\label{eq:EstimL2strongBIS}
\int_0^T \|  f   \omega  \Delta^{-\mu }  \|^2_{L^{2}(\OO)} \varphi^2 dt
\lesssim   \int_0^T \| f\omega  \|_{L^2(\OO)}^2 ( \varphi^2 + (\varphi')^2)  dt. 
\eeqn
We conclude by interpolating \eqref{eq:EstimL2LpBIS} and  \eqref{eq:EstimL2strongBIS}  as in \cite[Proposition~3.7]{carrapatoso2024kinetic}.
\end{proof}

  \subsection{Half ultracontractivity estimate using Nash's argument}

In the spirit of Nash work~\cite{MR0100158}, we  deduce a time pointwise gain of integrability estimate. 
\begin{cor}\label{cor:L1Lr} There exists   $\nu_1> 0$ such that 
\beqn\label{eq:corL1Lr}
  \| f(t,\cdot) \|_{L^2_\omega(\OO)}  \lesssim    {1 \over t^{\nu_1} } \| f_0 \|_{L^1_\omega(\OO)} , \quad \forall \, t \in (0,1).
\eeqn
\end{cor}

\begin{proof}[Proof of Corollary~\ref{cor:L1Lr}] 
For $\varphi \in C^1(0,T) \cap C_0(0,T)$, $T \in (0,1)$, using first an interpolation inequality and next an Holder inequality, we have 
\bean
 \int_0^T  (\varphi^2 + ( \varphi')^2)\|   f \|^2_{L^2_{\omega}} dt 
 &\le&
 \int_0^T  (1 +  ( \varphi'/\varphi)^2) \varphi^{2\theta} \|   f \|_{L^1_{\omega}}^{2\theta} \varphi^{2(1-\theta)} \|   f \|_{L^r_{\omega}}^{2(1-\theta)} dt 
\\
&\le&
\Bigl( \int_0^T  (1 +  ( \varphi'/\varphi)^2)^{1/\theta} \varphi^2  \|f \|_{L^1_{\omega}}^2 dt  \Bigr)^{\theta} 
\Bigl( \int_0^T   \varphi^2  \|f \|_{L^r_{\omega}}^2 dt  \Bigr)^{1-\theta} , 
\eean
with $r > 2$ given by Proposition~\ref{prop:L2Lr} and $\theta \in (0,1)$ such that $1/2=\theta+(1-\theta)/r$. Gathering this  estimate and the estimate of Proposition~\ref{prop:L2Lr}, 
we obtain after simplification 
\bean
 \int_0^T  (\varphi^2 + ( \varphi')^2)\|   f \|^2_{L^2_{\omega}} dt 
\lesssim   \int_0^T  (1 +  ( \varphi'/\varphi)^2)^{1/\theta} \varphi^2  \|f \|_{L^1_{\omega}}^2 dt.
\eean
Using the estimates 
$$
 \|  f (T,.)   \|_{L^{2}_{\omega}} \lesssim   \|  f(t,.)  \|_{L^{2}_{\omega}}, 
 \quad 
  \|  f(t,.)  \|_{L^{1}_{\omega}}  \lesssim    \|  f_0   \|_{L^{1}_{\omega}}, 
 $$
 for any $0 < t < T < 1$ established in Lemma~\ref{lem:GrowthLp},  we deduce that  
\bean
C_2(\varphi,T)  \|   f(T,\cdot) \|_{L^2_{\omega}}  
\lesssim C_1(\varphi,T)  \|f \|_{L^1_{\omega}}, 
\eean
with
$$ 
C_2(\varphi,T)^2 :=    \int_0^T  (\varphi^2 + ( \varphi')^2) dt, 
\quad
C_1(\varphi,T)^2 :=    \int_0^T  (1 +  ( \varphi'/\varphi)^2)^{1/\theta} \varphi^2    dt.
$$
We finally take $\varphi(t) := \varphi_0(t/T)$, $\varphi_0(s) = s^{1/\theta}(1-s)^{1/\theta}$, and we observe that 
$$
C_2(\varphi,T)^2 \ge C_2(\varphi_0,1)^2 T^{-1}, 
\quad
C_1(\varphi,T)^2 \le C_1(\varphi_0,1)^2 T^{-1-2/\theta}
$$
and $(1 +  ( \varphi_0'/\varphi_0)^2)^{1/\theta} \varphi_0^2 \in L^\infty(0,1)$, 
from what we immediately conclude that \eqref{eq:corL1Lr} holds with $\nu := 1/\theta$. 
\end{proof}

  \subsection{The other half ultracontractivity estimate using Moser's argument}

   With Corollary~\ref{cor:L1Lr}, we have made  
    half of the proof of Theorem~\ref{theo-Ultra}. In order to cover the other half of the way, we may proceed using one of De Giorgi, Nash or Moser approaches. 
   We briefly present the     Moser approach based on a iterated scheme.

\begin{lem}\label{lem:Moser-estimL2+}
Assume $\omega := e^{\alpha y}$, $\alpha > 0$.
There exists a constant $\widetilde\kappa \ge 0$  such that  any (nonnegative) solution $f $ to the linear VCk equation \eqref{eq:VCklinear}-\eqref{eq:LinearVCkBd} 
satisfies 
 \beqn\label{eq:lem:estimL2+}
 \int_0^T \varphi^2 \int_\OO  \Bigl\{ (f \widetilde\omega)^p \Bigl[ {J^2 \over \delta^{1/2} \langle y \rangle^2} +  \alpha \langle y \rangle \Bigr] + (\partial_y ( f \widetilde\omega)^{p/2} )^2 \Bigr\}  \le p \widetilde\kappa 
\int_0^T ( \varphi \varphi'_+ + \varphi^2)  \int_\OO  (f \widetilde\omega)^p, 
\eeqn 
 for any $p \in [2,\infty)$ and $\varphi \in C^1_c(0,T)$, where $\widetilde\kappa = \widetilde\kappa (a_*,y_*,a^*)$  
 and $ \widetilde \omega$ is defined in \eqref{eq:def-widetildeomega}.
\end{lem}

\begin{proof}[Proof of Lemma~\ref{lem:Moser-estimL2+}]  
We adapt the proofs of  Lemma~\ref{lem:GrowthLp} and Lemma~\ref{lem:estimL2+}. 
We set
$$
\widecheck\omega^p := \widetilde \omega^p \frak P ,  
$$
with now $\frak P$  defined by 
$$
\frakP := 1 -  \wp {\delta^{1/2} } {\delta' J \over \langle J_\xi \rangle^2}, \quad \wp := \frac12   \min \bigl( \frac1{\delta^*}, {y_* \over a^*} \frac1{\delta^\sharp} \Bigr), 
$$
where  we define $\delta_\sharp^2 := \sup [ \delta (\delta')^2 (v_E-v)^2]$ and  we recall that $\delta^2_* := \sup [\delta (\delta')^2]$. 
Proceeding exactly as during the  proof of Lemma~\ref{lem:GrowthLp}, we have 
  \bean
 \frac1p  {d \over dt} \int_\OO f^p \widecheck\omega^p
&=& - {4 (p-1) \over p^2}  \int_\OO \fraka (\partial_y (f\widecheck\omega)^{p/2})^2 +  \int_\OO  f^p \widecheck\omega^p \widecheck\varpi
\\
&& + \int_{\Sigma_0} [(1-\frac1p) K + (1-\frac2p) \fraka   {\partial_y\check\omega \over \check\omega} ] n_0     (f \widecheck\omega)^p   
-   \frac1p \int_{\Sigma_2} J f^p \widecheck\omega^p n_v,
\eean
with 
$$
\widecheck\varpi := 2(1- \frac{1}{p} ) \fraka \bigl( {\partial_y \check\omega \over \check\omega} \bigr)^2
+ (\frac2p - 1) \fraka {\partial^2_{yy} \check \omega \over \check\omega}
+ K  {\partial_y \check\omega \over \check\omega} + J  {\partial_v \check\omega \over \check\omega} 
+ ( \frac{1}{p} -1) \partial_y K 
+ ( \frac{1}{p} -1) \partial_v J. 
$$
 It is worth emphasizing that because of the definition of $\wp$, we have $\frakP \ge 1/2$ on $\OO$, we have $|\partial_y \frakP | =  \wp \delta^{1/2} |\delta'|(v_E-v) \le \wp/\delta_\sharp$ on $\Sigma_0$, and thus 
\bean
(1-\frac1p) K + (1-\frac2p) \fraka   {\partial_y\check\omega \over \check\omega}
&=& 
(1-\frac1p) \frakb + (1-\frac2p) \fraka   {\partial_y\frakP \over \frakP}
\\
&\ge&
(1-\frac1p) y_* + (\frac2p-1) a^* 2   {\wp \over\delta_\sharp} \ge 0
\eean
on $\Sigma_0$. As a consequence, the arguments in Lemma \ref{lem:GrowthLp} remain valid for the two terms involving the boundary sets $\Sigma_0$ and $\Sigma_2$ and these ones are thus   nonpositive. 
Also observing that 
$$
 -   \int_\OO \fraka (\partial_y (f\widecheck\omega)^{p/2})^2 
 \le - \frac12   \int_\OO \fraka (\partial_y (f\widetilde\omega)^{p/2})^2 \frakP +   \int_\OO \fraka   (f\widetilde\omega)^p  (\partial_y \frakP^{1/2})^2 , 
 $$
 we deduce that 
 \bean
 \frac1p  {d \over dt} \int_\OO f^p \widecheck\omega^p
&\le& - { (p-1) \over p^2}  \int_\OO \fraka (\partial_y (f\widetilde\omega)^{p/2})^2  +  \int_\OO  f^p \widetilde\omega^p \widetilde\varpi, 
\eean
with 
\bean
\widetilde\varpi &:=& \frakP \widecheck\varpi + \fraka     (\partial_y \frakP^{1/2})^2. 
\eean
Observing that 
$$
{\partial \check \omega \over \check \omega} = 
{\partial \tilde \omega \over \tilde \omega} + \frac1p
{\partial \frakP \over  \frakP}  
$$
and  that 
$$
{\partial^2_{yy} \check \omega \over \check \omega} 
= {\partial^2_{yy} \tilde \omega \over \tilde \omega} 
+  \frac2p {\partial_y \tilde \omega \over \tilde \omega}  {\partial_y \frakP \over  \frakP}  
+ \frac{1-p}{p^2}
{\partial_y \tilde \omega \over \tilde \omega}  {\partial_y \frakP \over  \frakP}  
+  \frac1p
{\partial_{yy}^2\frakP \over  \frakP},
$$
we get 
\bean
\widetilde\varpi &=&  \frakP \varpi  
+  \fraka  \Bigl\{ \frac2p {\partial_y \tilde \omega \over \tilde \omega}  \partial_y \frakP  
+ \Bigl(  \frac12+ \frac1p-\frac1{p^2}\Bigr)    {(\partial_y \frakP)^2 \over  \frakP}  
+ \frac{2-p}{p^2} \partial_{yy}^2\frakP  \Bigr\} + \frac1p K \partial_{y}\frakP +\frac1p J \partial_{v}\frakP,
\eean
 where  $\varpi$  is defined by \eqref{def:varpi} and thus uniformly bounded in $p$. 
 Moreover, as
  \bean
(\frac{\partial_y\frakP}{\frakP})^2\sim\frac{1}{y^4}, \
\frac{\partial_y\widetilde\omega}{\widetilde\omega}\frac{\partial_y\frakP}{\frakP}\sim\frac{1}{y^2}, \
\frac{\partial_{yy}\frakP^{1/2}}{\frakP^{1/2}}\sim\frac{1}{y^2}, \
K(\frac12\frac{\partial_{y}\frakP}{\frakP})\sim\frac{1}{y^2}
\eean
and 
$$
\partial_v \frakP =  - \frac\wp2   {(\delta')^2 \over \delta^{1/2}} {  J \over \langle J \rangle^2} + {\phi \over \langle J \rangle},
\quad  \phi \in L^\infty(\OO), 
$$ 
we conclude that 
\bean
\widetilde\varpi -  \frac1p J  \partial_v \frakP
\sim  \varpi    \frakP \sim   -  \frakP \alpha y
\eean
when $y \to \infty$ and any value of $\alpha$ uniformly in $p$. Because  $\langle J \rangle \simeq \langle y \rangle$, the previous estimate gives
\bean
 {d \over dt} \int_\OO f^p \widecheck\omega^p
 +   \int_\OO f^p \widetilde\omega^p \Bigl( {1 \over 2} {1 \over \delta^{1/2}} {J^2 \over \langle y \rangle^2}  + \frac{\alpha}{2}   { \langle y \rangle} \Bigr)
  +   \int_\OO a_* (\partial_y (f\widetilde\omega)^{p/2})^2   \lesssim p \int_\OO f^p \widetilde\omega^p. 
\eean
We conclude by multiplying the above equation by $\varphi^2$ and integrating in the time variable. 
\end{proof}

\begin{lem}\label{lem:gL2estim+}
 Assume $\omega := e^{\alpha y}$, $\alpha > 0$. There exists  
a constant $C = C(a_*,a^*) \in (0,\infty)$ such that 
any solution $f$ to the linear VCk equation   \eqref{eq:VCklinear}, \eqref{eq:LinearVCkBd}   satisfies 
$$
\int_0^T \varphi^2 \int_\OO  (f\widetilde\omega)^p \Bigl[  {1  \over \delta^{1/27}} + \langle y \rangle  \Bigr]  \le  p C
\int_0^T (\varphi \varphi'_+  + \varphi^2) \int_\OO  (f\widetilde\omega)^p, 
$$
for any $\varphi \in C^1_c(0,T)$ and any $p \in [2,\infty)$. 
\end{lem}

\begin{proof}[Proof of Lemma~\ref{lem:gL2estim+}]
The proof is straightforward by using Lemma~\ref{lem:Moser-estimL2+} and by following step by step the proof of Lemma~\ref{cor1:L2estim+} and  Proposition~\ref{prop:estimL2++}. 
\end{proof}

Similarly as in Section~\ref{subseq:GainL2}, for a weight function $\omega := e^{\alpha y}$, $\alpha > 0$, a solution  $f $ to the linear VCk equation \eqref{eq:VCklinear}-\eqref{eq:LinearVCkBd} and an exponent $p \ge 2$, we define 
\beqn\label{def:q-h-barh}
 q := p/2, \quad h := (f\widetilde\omega)^q, \quad \bar h :=h\varphi(t)\chi(v)\omega_0(y), 
\eeqn
with  $ \varphi \in C^1_c((0,T))$, $\chi \in C^1_c((0,v_F))$, $0 \le \varphi, \chi \le 1$, $\omega_0=\frac{1}{\langle y\rangle}$. The function $\bar h$ is  a solution to the modified VCk equation 
 \beqn\label{eq:KFP-interior-barh}
 \partial_t \bar h+ \hat J \partial_v \bar h+   \check K \partial_y \bar h - \fraka \partial_{yy}^2 \bar h + \alpha \langle y \rangle \bar h=  H  -\fraka  \chi \varphi \omega_0 q(q-1)f^{q-2}(\partial_y f)^2\widetilde\omega^q
\eeqn
 with coefficients 
\bean
&&\hat J := J(v,y), \ \hbox{if} \ y > 0, \quad
\hat J  :=  J(v,-y), \ \hbox{if}  \ y < 0, 
\\
&&\check K := K(v,y), \ \hbox{if} \ y > 0, \quad
\check K :=  - K(v,-y), \ \hbox{if}  \ y < 0, 
\eean
and  source term
\bean
H &:=&  
 (\partial_t \varphi) \chi \omega_0  h  
+  \hat J \varphi ( \partial_v \chi) \omega_0 h 
 - q(\partial_v  \hat J)  \varphi   \chi \omega_0 h  
 +   \check  K \varphi \chi ( \partial_y \omega_0)  h 
 - q(\partial_y \check  K)  \varphi   \chi \omega_0 h  
 \\&&
 \qquad
  - \fraka    \varphi    \chi  (\partial_{yy}^2  \omega_0) h 
- 2 \fraka   \varphi    \chi  ( \partial_y  \omega_0 +   q\omega_0\frac{\partial_y\tilde\omega}{\tilde \omega}) (\partial_y  h) 
  + \alpha \langle y \rangle   \varphi   \chi \omega_0 h \\
&&\qquad
 +q(\frac{\partial_v\tilde\omega J}{\tilde \omega}+\frac{\partial_y \tilde\omega K}{\tilde \omega})\varphi   \chi \omega_0 h
+\fraka(q+1)q h (\frac{\partial_y\tilde\omega }{\tilde \omega})^2 \varphi   \chi \omega_0-   \fraka qh \frac{\partial_{yy}^2\tilde\omega }{\tilde \omega} \varphi   \chi \omega_0.
\eean

We note that the associated homogeneous equation is again \eqref{eq:KFP-interior-g}. Moreover, since
\begin{equation}\label{eq:Hbound}    
|H| \lesssim |\partial_t \varphi h| +  |\varphi \partial_v \chi h  | + q^2|\varphi  h  | +  q|\varphi  \partial_y(h )|, 
\end{equation}
and the term $-\fraka q(q-1)f^{q-2}(\partial f)^2\omega^p$ is negative, a result analogous to Proposition~\ref{prop:L2LpInterior-barf} holds for the function $\bar h$. 

\begin{prop}\label{prop:L2LpInterior-barh} 
The   function $\bar h$ defined in \eqref{def:q-h-barh}  satisfies 
\beqn\label{eq-prop:L2LpInterior-barh}
 \| \bar h\|^2_{L^{5/2}(\UU) } \lesssim q^3\|  \chi \|_{W^{1,\infty}}^2 \int_0^T \|   h \|^2_{L^2} (\varphi^2 +  (\varphi')^2)dt , 
\eeqn
for any $T \in (0,1)$, $\varphi \in C^1_c(0,T)$ and $\chi \in C^1_c(0,v_F)$. 
\end{prop}

\begin{proof}[Proof of Proposition~\ref{prop:L2LpInterior-barh}]
    As in Proposition \ref{prop:L2LpInterior-barf}, we may use Duhamel formula to obtain
    $$
    \bar h=\int_0^t S_{t-s}^0H_s ds - q(q-1)\int_0^t S_{t-s}^0 (\fraka f^{q-2}(\partial_y f)^2\omega^q)(s) ds, 
    $$
    for $t\in(0,1)$. 
   We can ignore the second term since it is negative, which allows us to proceed as before to obtain
   $$
   \|\bar h\|_{L^{5/2}(\R^2)}\le t^{1/10}\|H\|_{L^2(U)}.
   $$
 We use then Proposition \ref{lem:gL2estim+} and observation \eqref{eq:Hbound}, to get
 $$\|\bar h\|_{L^{5/2}(\R^2)}\le t^{1/10}q^3\|  \chi \|_{W^{1,\infty}}^2 \int_0^T \|   h \|^2_{L^2} (\varphi^2 +  (\varphi')^2)dt.$$
From where, we deduce 
    $$
    \|\bar h\|_{L^{5/2}(\UU)}\le q^3\|  \chi \|_{W^{1,\infty}}^2 \int_0^T \|   h \|^2_{L^2} (\varphi^2 +  (\varphi')^2)dt (\int_0^T t^{1/4})^{2/5}dt,
    $$
    and then immediately conclude.
\end{proof}

\begin{prop}\label{prop:L2Lrh}  There exists $r > 2$ such that the function $h$ defined in \eqref{def:q-h-barh}  satisfies 
$$
  \| \varphi h \|_{L^r(\UU)}^2   \lesssim  q^3\| (\varphi + |\varphi'|)   h  \|_{L^2(\UU)}^2, 
$$
for any $\varphi \in C^1_c((0,T))$. 
\end{prop}
\begin{proof}[Proof of Proposition~\ref{prop:L2Lrh}] On the one hand, from Proposition~\ref{prop:L2LpInterior-barh} and arguing as during the 
proof of Proposition~\ref{prop:L2LpInterior-barf}, for any $\beta > 5/2 $ and setting  $\Delta :=  \delta^{2\beta/5} / \langle y \rangle$, we have 
   $$
  \| \varphi \Delta h \|_{L^{5/2}(\UU)}^2   \lesssim  q^3\| (\varphi + |\varphi'|)   h  \|_{L^2(\UU)}. 
$$
On the other hand, arguing as during proof of Proposition~\ref{prop:L2Lr}, we may rewrite the conclusion of Lemma~\ref{lem:gL2estim+} as
   $$
  \| \varphi \,\Delta^{-\mu} h \|_{L^{2}(\UU)}^2   \lesssim q \| (\varphi + |\varphi'|)   h  \|_{L^2(\UU)}^2,  
$$
for some suitable $\mu > 0$. 
We cocnclude thanks to an interpolation argument.
 \end{proof}

\begin{prop}\label{prop:Moser}
    Assume $\omega := e^{\alpha y}$, $\alpha > 0$. 
Any solution $f $ to the linear VCk equation  \eqref{eq:VCklinear}-\eqref{eq:LinearVCkBd}  
 satisfies 
  \begin{equation*}
\left\|f_{T}  \omega\right\|_{L^{\infty}} \lesssim T^{-\frac{r}{r-2}} 
\left\|f_{0} \omega\right\|_{L^{2}}, \quad \forall \, T \in (0,1), 
\end{equation*}
where $r > 2$ is defined in Proposition~\ref{prop:L2Lrh}.
\end{prop}

\begin{proof}[Proof of Proposition~\ref{prop:Moser}]
We follow Moser's iterative method and thus  introduce the following sequences
$$
t_k :=\frac{T}{2}-\frac{T}{2^k}, \quad\UU_k: =(t_k,T]\times \OO, \quad p_k: =(r/2)^k2,
$$
where $r$ is defined in Proposition~\ref{prop:L2Lrh} and $T\in(0,1]$. 
 
 \smallskip\noindent
 {\sl Step 1.} In Proposition \ref{prop:L2Lrh}, we take $q=p_k/2$ and $\varphi \in C^1_c(0,1)$ such that $ {\bf 1}_{(t_{k+1},T)} \le \varphi \le {\bf 1}_{(t_k,T)}$
 and  $|\varphi^\prime|\le 2 \frac{1}{t_{k+1}-t_k}$. 
 We obtain then 
   \begin{equation*}
       \ \| (f\widetilde\omega)^{p_k/2}  \|_{L^r(\UU_{k+1})}^2  \lesssim p_k^3 \frac{1}{(t_{k+1}-t_k)^2}  \|  (f\widetilde\omega)^{p_k/2}  \|_{L^2(\UU_k)}^2.
   \end{equation*}
   Noticing that $p_{k+1}=p_k\frac{r}{2}$, we may rewrite equivalently the above estimate as 
   \beqn\label{eq:propMoser-step1}
\|f\widetilde\omega\|_{L^{p_{k+1}}(\UU_{k+1})}\le \left(C p_k^3 \frac{2^{2k}} { T^2} \right)^{\frac{1}{p_k}}   \|f\widetilde\omega\|_{L^{p_k}(\UU_{k})}.
\eeqn

    \smallskip\noindent
 {\sl Step 2.} 
 Observing that
$$
\sum_{k=1}^{\infty} \frac{1}{p_{k}}=\frac{1}{2} \sum_{j=0}^{\infty} \left(\frac{2}{r}\right)^j=\frac{r}{r-2},
$$
$$
\sum_{k=1}^{\infty} \frac{k}{p_{k}}=\frac{1}{2} \sum_{j=0}^{\infty} j\left(\frac{2}{r}\right)^j=c_1<\infty,
$$
and
$$
\prod_{k=1}^{\infty} p_j^{3/p_j}=e^{3\sum_{k=1}^{\infty} \frac{\log p_j}{p_{j}}}=e^{3\log{( \frac{r}{2})}\sum_{k=1}^{\infty} j(\frac{2}{r})^j}=e^{3\log{( \frac{r}{2})}c_1},
$$
we deduce that
$$
\prod_{j=1}^{k}\left(Cp_j^3 \frac{2^{2j}}{T^2}\right)^{1 / p_{j}}\lesssim T^{-2\frac{r}{r-2}}. 
$$
As a consequence, we have 
 $$
\begin{aligned}
\|f \widetilde\omega\|_{L^{\infty}\left(\mathcal{U}_{\infty}\right)} 
& \leq \liminf _{k \rightarrow \infty}\|f  \widetilde\omega\|_{L^{p_{k}} (\mathcal{U}_{k})}  \\
& \leq \liminf _{k \rightarrow \infty} \prod_{j=1}^{k}\left(Cp_j^3 \frac{2^{2j}}{T}\right)^{1 / p_{j}}\|f  \widetilde\omega\|_{L^{p_{1}}\left(\mathcal{U}_{1}\right)} \\
& \lesssim T^{-2\frac{r}{r-2}}\|f \widetilde\omega\|_{L^{2}\left(\mathcal{U}_{1}\right)}.
\end{aligned}
$$
We conclude by using Lemma \ref{lem:GrowthLp} with $p=2$ in order to bound the RHS term. 
\end{proof}

We finally come to the 
\begin{proof}[Proof of Theorem~\ref{theo-Ultra}]
We may reduce the proof to the  case $s=0$ by using a mere time translation. From \eqref{eq:corL1Lr} and Proposition~\ref{prop:Moser}, we deduce that 
\beqn\label{eq:corL1Linfty}
  \| f(t,\cdot) \|_{L^\infty_\omega(\OO)}  \lesssim    {1 \over t^{\nu} } \| f_0 \|_{L^1_\omega(\OO)} , \quad \forall \, t \in (0,1), 
\eeqn
with $\nu := \nu_1 + 2r/(r-2)$. 
Estimate \eqref{eq:theo-Ultra} is then just a combination of \eqref{eq:corL1Linfty} used on the time interval $(0,1)$ and the estimate
\eqref{eq:lem:GrowthLp} for $p=\infty$ used on the time interval $[1,\infty)$. 
\end{proof}

\Black
 
 \section{Interior Holder, Harnack and compacteness estimates} 
  \label{sec:Holder&compact}

We briefly explain how we can establsih that any solution $f$ to the linear kinetic equation \eqref{eq:VCklinear} enjoys some further regularity estimates.

 \begin{theo} \label{theo-L1Calpha-f}
 For any $0 < t_0 < t_1 < \infty$ and any relatively compact region $O \subset\subset \OO$, there exist two constants $C,\beta > 0$ such that 
 any solution $f$  to the linear  Voltage-Conductance kinetic equation  \eqref{eq:VCklinear}, \eqref{eq:LinearVCkBd} satisfies 
 $$
 \| f \|_{C^\beta((t_0,t_1) \times O)} \le C \| f_0 \|_{L^1_{\omega_1}}.
 $$
 \end{theo}

\smallskip\noindent
{\sl Proof of Theorem~\ref{theo-L1Calpha-f}.}  
Introducing the localization $\bar f$  as in \eqref{eq:KFP-interior-barf}, we write
\beqn\label{eq:VCkBIS}
\partial_t \bar f  + J \partial_v \bar f  + K \partial_y\bar f -   \fraka \partial_{yy}^2 \bar  f = F - \alpha \langle y \rangle \bar f  
=: {\frak F}
\quad\text{in}\quad (0,\infty) \times \OO. 
\eeqn
We introduce the change of variables
$$
\bar g(t,w,z) = \bar f(t,v,y), \quad  {\frak G}(t,w,z) =  {\frak F} (t,v,y),
$$
with $w=w(v,y)$ and $z=z(v,y)$. Assuming 
\beqn\label{eq:z&wOFy}
\partial_y z = 1, \quad \partial_y w = 0,
\eeqn
we straightforwardly obtain 
\bean
 {\frak G} 
 &=& \partial_t \bar f  + J \partial_v \bar f  + K \partial_y \bar f - \fraka \partial_{yy}^2  \bar f
\\
 &=& \partial_t \bar g   + J \partial_vw \partial_w \bar g + J \partial_vz \partial_z \bar g  +   K   \partial_z \bar g   - \fraka \partial_{zz}^2 \bar g.
\eean
The previous conditions on $w$ and $z$  are equivalent to   
\beqn\label{eq:defwz}
w := \phi(v), \quad z := y + \psi(v), 
\eeqn
and we look for $\phi$ and $\psi$ such that 
$$
J \partial_vw = z, 
$$
or in other words
$$
(y (v_E-v) - y_L v) \phi'(v) = y + \psi(v).
$$
A convenient choice is 
\beqn\label{eq:defphipsi}
\phi(v) := - \log (v_E-v), \quad \psi(v) := -  {y_L v \over v_E-v}. 
\eeqn
As a conclusion, with this change of variables, the function $\bar g$ is a solution to 
\beqn\label{eq:VCkTER}
\partial_t \bar g  + z \partial_w \bar g + Q \partial_z \bar g  - \fraka \partial_{zz}^2 \bar g  = {\frak G}, 
\eeqn
with $Q := J \psi' + K$. 
Because of the compact support condition on $\bar g$, we may for instance understand 
the equation on $\R_+ \times \R^2$ by defining $Q \in W^{1,\infty}(\R^2)$ arbitrarily outside of the support of $\bar g$. 
By definition and using Theorem~\ref{theo-Ultra}, we have $\bar g, {\frak G} \in L^\infty(\R_+ \times \R^2)$. 

\smallskip
We introduce a mollifier $(\rho_\eps)$ in $\DD(\R^2)$ and the regularization sequence $\bar g_\eps := \bar g *_{wz} \rho_\eps$ which satisfies 
$$
\partial_t \bar g_\eps  + z \partial_w \bar g_\eps + Q \partial_z \bar g_\eps  - \fraka \partial_{zz}^2 \bar g_\eps  = {\frak G}_\eps, 
$$
with ${\frak G}_\eps := {\frak G} *  \rho_\eps + [\rho_\eps,z \partial_w+Q \partial_z ] \bar g$. Using \cite[Lemma~II.1]{MR1022305}, we have ${\frak G}_\eps \to {\frak G}$ in any $L^p(\R^2)$ space, 
$1 \le p < \infty$. We may multiply the above equation by $ \bar g_\eps$ and integrate in all variables. Using the Green formula and the fact that the boundary term vanishes because of the compact
support of $\bar g_\eps$, we get 
$$
\int_\UU (\partial_{z} \bar g_\eps)^2 = \int_\UU [ \frac 12 \bar g_\eps^2  \partial_z Q  + \bar g_\eps  {\frak G}_\eps] \lesssim \| \bar g \|^2_{L^2(\UU)} + \| {\frak G} \|_{L^2(\UU)}^2.
$$
Passing to the limit, we deduce $\partial_z \bar g \in L^2(\UU)$. 
 We may thus apply  \cite[Theorem~1.3]{zbMATH07050183}, from which we immediately conclude that $\bar g \in C^\beta$, for some $\beta \in (0,1)$,  so that  $\bar f$ and $f$ are also $C^\beta$ functions.
 \qed

\bigskip

 Using the same change of variables as above and the Harnack inequality for the standard kinetic Fokker-Planck  equation as developped in  \cite{zbMATH07050183,MR4181953,MR4412380,CGMM**}, we deduce the Harnack inequality for the VCk equation  \eqref{eq:VCklinear}.

 \begin{theo}\label{theo:Harnack}
 For $\eps \in (0,1)$, we define $\OO_\eps := \{ (v,y) \in \OO; \ \delta(v) > \eps, \ y < 1/\eps\}$. 
For every $T > T_0>0$,  and $\eps>0$, there exists $C$ such that any solution $ 0 \le f \in L^\infty(\UU)$ to the linear  VCk equation  \eqref{eq:VCklinear}  satisfies 
 \beqn\label{eq:Harnack}
  \sup_{\OO_\eps} f_{T_0} \le C  \inf_{\OO_{\eps}} f_{T}. 
\eeqn
  \end{theo} 
  
  \begin{proof}[Proof of Theorem~\ref{theo:Harnack}] We define $g(t,w,z) =   f(t,v,y)$ with the previous change of variables \eqref{eq:defwz} and \eqref{eq:defphipsi}. The function $g$ thus  satisfies 
  $$
  \partial_t g  + z \partial_w g   - \fraka \partial_{zz}^2 g + Q \partial_z g  + A g = 0 
$$
on $U := (0,\infty) \times O$, with $Q$ defined during the proof of Theorem~\ref{theo-L1Calpha-f},  $A := \partial_v J + \partial_y K$ 
and $O := \Xi(\OO)$,  $\Xi(v,y) := (\phi(v),y+\psi(v))$.  Because $\Xi$ is a homeomorphism, the set $O_\eps := \Xi(\OO_\eps)$ is relatively compact in $O$. 
We may thus use the Harnack inequality established in \cite[Theorem~6.1]{CGMM**} which tells us that there exists a cosntant $C = C(\eps,T,T_0)$ such that 
$$
  \sup_{O_\eps}  g_{T_0} \le C  \inf_{O_{\eps}} g_{T}, 
$$
from what \eqref{eq:Harnack} immediately follows. 
 \end{proof}

 We conclude this section by formulating the following compactness result.

 \begin{theo} \label{theo-compactness}
 Consider a sequence $(f_n)$ bounded in $L^\infty_\omega(\UU) \cap \HH_\omega$ such that 
\bean
&& \partial_t f_n + \partial_v(Jf_n) + \partial_y(K_n f_n) - \fraka_n \partial^2_{yy} f_n = 0  \ \text{ in } \ \UU, 
\\
&& \RRR_{\fraka_n,K_n} \gamma f_n = 0 \ \text{ on }\  \Gamma , \quad f_n(0) = f_{0n} \ \text{ in } \ \OO, 
\eean
with $K_n := \frakb_n - y$,  $(\fraka_n)$ and $(\frakb_n)$ bounded in $L^\infty(0,T)$, $\fraka_n \ge a_*$, $\frakb_n \ge y_*$  and $f_{0n} \to f_0$ in $L^1(\OO)$. 
There exist $f \in L^\infty_\omega(\UU) \cap \HH_\omega$ and $\fraka,\frakb \in L^\infty(0,T)$ 
such that, up to the extraction of a subsequence, $f_n \to f$ strongly in $L^1(\UU)$,
$\fraka_n \wto \fraka$  and $\frakb_n \wto \frakb$ weakly in $L^\infty(0,T)$ 
and 
\bean
&& \partial_t f + \partial_y(Jf) + \partial_y(K f) - \fraka \partial^2_{yy} f = 0  \ \text{ in } \ \UU, 
\\
&& \RRR_{\fraka,K} \gamma f = 0 \ \text{ on }\  \Gamma , \quad f(0) = f_{0} \ \text{ in } \ \OO, 
\eean
in the variational sense with  $K :=  \frakb - y$.
 \end{theo}

\begin{proof}[Proof of Theorem~\ref{theo-compactness}]
{\sl Step 1.} We define the sequence $(\bar f_n)$ as in \eqref{eq:KFP-interior-barf} for a localization function $\psi \in C^1_c(\UU)$. From Theorem~\ref{theo-L1Calpha-f}, we have 
$$
\| \bar f_n \|_{C^\beta(U)} \le C
$$
for some $\beta \in (0,1)$ and any $U \subset\subset \UU$. From the Ascoli theorem and a Cantor diagonal process, the sequence $(\bar f_n)$ and next the sequence $(f_n)$ 
belong to a strong compact set of $L^1(\UU)$. Up to the extraction of a subsequence, there exists thus $f \in L^\infty_\omega(\UU) \cap \HHH_\omega$ such that $f_n \to f$ strongly in $L^1(\UU)$
and $\partial_y f_n \to \partial_y f$ weakly in  $L^2(\UU)$. We may also assume that, up to the extraction of the same subsequence,   $\fraka_n \wto \fraka$  and $\frakb_n \wto \frakb$ weakly in $L^\infty(0,T)$.

\smallskip\noindent
{\sl Step 2.} We claim that for any $ \psi \in C_c^1(\UU \cup \Gamma_-)$, we have 
\beqn\label{eq:compactnessStrongMean}
\int_\OO \partial_y f_n  \partial_y \psi \to \int_\OO \partial_y f  \partial_y \psi 
\ \hbox{ strongly } \ L^1(0,T).
\eeqn
We first assume $ \psi \in C_c^4(\UU)$ and we denote 
$$
\Psi_n := \int_\OO \partial_y f_n  \partial_y \psi.
$$
We clearly have $(\Psi_n)$ is bounded in $L^2(0,T)$ from the fact that $(f_n)$ is bounded in $\HHH_\omega$. On the other hand, we compute 
\bean
{d \over dt} \Psi_n 
&=& \int_\OO \partial_y f_n  \partial_t\partial_y \psi  -  \int_\OO \partial_t f_n  \partial^2_{yy} \psi 
\\
&=& 
\int_\OO   f_n  [ J \partial^3_{vyy} \psi - K_n \partial^2_{yyy} \psi - \fraka_n \partial^4_{yyyy} \psi - \partial^3_{tyy} \psi]  , 
\eean
where the RHS is bounded in $L^\infty(0,T)$. From the Rellich theorem, we deduce that $(\Psi_n)$ is relatively compact in $L^1(0,T)$ and thus \eqref{eq:compactnessStrongMean} holds. 
For a general function $ \psi \in C_c^1(\UU \cup \Gamma_-)$, we may introduce an approximation family $(\psi_\eps)$ of $C^4(\UU)$ such that $\|\psi - \psi_\eps \|_{W^{1,\infty}} \le \eps$ for any $\eps > 0$. 
We obvious notations, we have already established that the associated sequence $(\Psi_{n,\eps})_{n \ge 1}$ is relatively compact in $L^1(0,T)$ for any $\eps > 0$. On the other hand, we compute 
$$
\| \Psi_{n,\eps} - \Psi_n \|_{L^1(0,T)} \le \int_\UU |\partial_y f_n (\partial_y \psi - \partial_y \psi_\eps) | \le \| f_n \|_{\HHH} \| \partial_y \psi - \partial_y \psi_\eps \| \lesssim \eps
$$
for any $n \ge 1$. That precisely mean that the sequence $(\Psi_n)$ is relatively pre-compact in $L^1(0,T)$, from what we classically deduce \eqref{eq:compactnessStrongMean}.

\smallskip\noindent
{\sl Step 3.} 
 Proceeding as in Step 7 of the proof of Theorem~\ref{theo:existL2}, we also know that $\gamma f_n \wto \gamma f$ weakly in $L^\infty(\Gamma)$. From the variational formulation \eqref{eq:var} of $f_n$, we have 
\beqn\label{eq:var-ffn}
 \int_\UU   f_n (- \partial_t  - J \partial_v)  \psi + \int_\UU (\fraka_n \partial_y f_n  - K_n  )  \partial_y  \psi
= \int_0^T \!\! \int_{y_F}^\infty   (J \gamma f_n)(t,v_F,y)  \psi(t,0,y) 
+ \int_\OO f_{0n} \psi(0,\cdot) 
\eeqn
for any $ \psi \in C_c^1(\UU \cup \Gamma_-)$. Because of the  strong  convergence properties of $(f_n)$ and $(\Psi_n)$ as well as  the weak convergence properties of $(\fraka_n)$, $(K_n)$ and $(\gamma f_n)$, we may pass to the limit in the above formulation, and we get that $f$ satisfies the variational formulation \eqref{eq:var}. 
\end{proof}

\section{Existence result for the nonlinear problem}  
\label{sec:ExistenceNL}

In this section, we present the proof of Theorem~\ref{theo-Exists} that we split into four steps

\medskip\noindent
{\sl Step 1.} 
We fix an admissible weight function $\omega$,  a constant $a^* > 2 (a_*+y_*)$ and   $\cc, N^* > 0$ such that $N^*(2\cc+2\cc^2) \le a^*$. 
For any function $0 \le N \in L^\infty(\R_+)$ such that $\| N \|_{L^\infty}  \le N^*$, 
we define the three functions $\fraka := a_* + \cc^2 N \in [a_*,a^*]$,  $\frakb := y_* + \cc N \in [y_*,a^*]$ and $K := \frakb-y$. 
We introduce the splitting 
\beqn\label{eq:splittingL=A+B}
\LL  := \LLL_{\fraka,K}  = \AA + \BB, \quad \AA f := M \chi_R f,
\eeqn
From the proof of Lemma~\ref{lem:GrowthLp}, we see that we may choose  $M = M(a^*) > 0$ and $R = R(a^*) > 0$ large enough in such a way that 
$$
\varpi^\BB_{\omega} = \varpi^\LL_{\omega}-M\chi_R \le - 1 
$$
and for $\omega := e^{\alpha y}$, $\alpha  > 0$, 
$$
\varpi^\BB_{\omega_\alpha} = \varpi^\LL_{\omega_\alpha} -M\chi_R \le - 1 - \alpha' y, \quad \alpha' \in (0,\alpha).
$$
We denote by $S_\BB(t,s)$ and $S_\LL(t,s)$ the associated semigroups and we use the shorthand $S_\BB(t) := S_\BB(t,0)$, $S_\LL(t) := S_\LL(t,0)$.
Repeating the proof of Lemma~\ref{lem:GrowthLp}, we know that  
\beqn\label{eq:DecaySBN}
\| S_{\BB}(t,s) f_s \|_{L^p_{\tilde\omega}} \le e^{-(t-s)} \| f_s \|_{L^p_{\tilde\omega}}, 
\quad \forall \, t \ge s \ge 0.
\eeqn
Let us introduce $\omega_\alpha = e^{\alpha y}$, $\alpha > 0$, such that $\omega_\alpha \ge \omega$.
Repeating the proof of Theorem~\ref{theo-Ultra}, we also have
\beqn\label{eq:UltraSBN}
  \| S_{\BB}(t,s) f_s \|_{L^\infty_{\omega_\alpha}} \le C_1 {e^{-(t-s)} \over (t-s)^\nu} \| f_s \|_{L^1_{ \omega_\alpha}}, 
\quad \forall \, t \ge s \ge 0, 
\eeqn
for a constant $C_1 \in (0,\infty)$ only dependent of $a^*$. 
Introducing the notations
$$
(U \star V)(t,s) := \int_s^t U(t,\tau) V(\tau,s) d\tau, 
$$
$U^{\star 1} = U$, $U^{\star k} = U^{\star (k-1)} \star U$, 
the Duhamel identity writes 
\beqn\label{eq:DuhamelIterate}
S_\LL = V + W \star S_\LL, \quad V := S_\BB + \dots +(S_\BB \AA)^{\star (N-1)} \star S_\BB, 
\quad W :=  (S_\BB \AA)^{\star N} .
\eeqn
As a consequence of \eqref{eq:DecaySBN}, \eqref{eq:UltraSBN} and the fact that $\AA : L^1 \to L^1_{\omega_\alpha}$, we have 
\beqn\label{eq:V&Wbd}
\| V (t) \|_{\BBB(L^\infty_\omega)} \le C_V e^{-t}, 
\quad
\| W (t,s) \|_{\BBB(L^1,L^\infty_\omega)} \le C_W e^{-(t-s)}.
\eeqn
For any $f_0 \in L^\infty_\omega$, we deduce 
\bean
\| S_\LL(t) f_0 \|_{L^\infty_\omega}  
&\le&  C_V e^{-t} \| f_0 \|_{L^\infty_\omega} + \int_0^t C_W e^{-(t-s)} \| S_\LL (s) f_0 \|_{L^1} ds
\\
&\le&  C_V e^{-t} \| f_0 \|_{L^\infty_\omega} + \int_0^t C_W e^{-(t-s)} ds \|  f_0 \|_{L^1}, 
\eean
and thus 
\beqn\label{eq:SLbdLinfty}
\| S_\LL(t) f_0 \|_{L^\infty_\omega}   \le  C_2 \| f_0 \|_{L^\infty_\omega},  
\eeqn
for a constant $C_2 \in (1,\infty)$ only dependent of $a^*$. 

\medskip\noindent
{\sl Step 2.} We set 
$$
C_3 := \int_{y_F} J(v_F,y) \omega^{-1} dy, \quad \eta^* := a^*/(2 C_2 C_3), 
$$
and
$$
\frak N := \{ N \in L^\infty(0,T); \, 0 \le N \le N^* \}.
$$
For $0 \le F_0 \in L^\infty_\omega$, we define $f = f_{N}(t) := S_{\LL}(t) F_0$ which is then the solution to    the linear  VCk equation 
\beqn\label{eq:fN}
\partial_t f = \LLL_{\fraka,K} f \ \hbox{in}\ (0,\infty) \times \OO,  \quad \RRR_{\fraka,K} f = 0 \ \hbox{in}\ (0,\infty) \times \Sigma, \quad f(0) = F_0 
\ \hbox{in}\  \OO.
\eeqn
Observing that 
\bean
0 \le \NN(f_{N}) &\le& C_3 \| \gamma f_N \|_{L^\infty_\omega} \le C_3 \|   f_N \|_{L^\infty_\omega} 
\\
&\le& C_3 C_2 \| F_0 \|_{L^\infty_\omega},  
\eean
 we have thus 
 $$
\Lambda : \frak N \to \frak N, \quad N \mapsto \Lambda (N) := \NN(f_{N}) , 
$$
under the smallness condition $ \| F_0 \|_{L^\infty_\omega} (\cc +\cc^2)\le \eta^*$.

\medskip\noindent
{\sl Step 3.} 
We endow $\frak N$ with the weak $L^2$ convergence. 
We claim that $\Lambda$ is continuous. Take indeed $(N_\ell)$ a sequence of $\frak N$ such that $N_\ell \wto N$ in $L^2$, 
and denote by $f_\ell$ the corresponding solution of \eqref{eq:fN} so that $(f_\ell)$ is bounded in $\frak F \cap \HHH$ with 
$$
\frak F := \{ f \in L^\infty_\omega; \, f \ge 0, \, \| f \|_{L^1} = \| F_0 \|_{L^1}, \, \| f \|_{L^\infty_\omega} \le C_2 \| F_0 \|_{L^\infty_\omega} \}.
$$
Thanks to Theorem~\ref{theo-compactness}, there exists a subsequence  $(f_{\ell_n})$ such that $f_{\ell_n} \to f$, where $f \in \frak F \cap \HHH$ is a solution to \eqref{eq:fN} associated to $N$. By uniqueness of $f$ in Theorem~\ref{theo:existL2}, it is the full sequence $(f_\ell)$ which converges to $f$. On the other hand, because $(\gamma f_\ell)$ is bounded in $L^\infty_\omega(\Gamma)$, there exists a subsequence  $(\gamma f_{\ell_n})$ such that $\gamma f_{\ell_n} \to \frak g$ weakly in $L^\infty_\omega(\Gamma)$. 
Passing to the limit in the Green formula which defines the trace function, we get that $ \frak g = \gamma f$, and thus it is the full sequence   $(\gamma f_{\ell})$ which converges to $\gamma f$. We thus deduce that 
$$
\Lambda(N_\ell) = \NN(  \gamma f _\ell ) \wto  \NN(  \gamma f  )  = \Lambda(N), 
$$
so that $\Lambda$ is indeed continuous. 

\medskip\noindent
{\sl Step 4.} The set $\frak N$ being obviously convex and compact for the weak $L^2$ convergence, we may  use the Schauder-Tykonov theorem which claims in that situation that there exists $\bar N$ such that $\Lambda (\bar N) = \bar N$. The function $F := f_{\bar N}$ is thus a solution to the nonlinear problem \eqref{eq:VCkt=0},  \eqref{eq:VCktBd1},  \eqref{eq:VCktBd2},  \eqref{eq:VCktBd3}. 
 \qed

\section{Doblin-Harris Theorem in a Banach lattice}  
\label{sec:DHtheorem}

We formulate a general abstract constructive Doblin-Harris theorem in the spirit of the ones presented in the recent works \cite[Section~6]{sanchez2023kreinrutman} and  \cite[Theorem~7.1]{CGMM**}, see also \cite{MR2857021,zbMATH07654553} for similar results and approaches in  the classical probability measures framework.
 
 The proof is a consequence of the previous estimates and of some  Doblin-Harris techniques developed in \cite{zbMATH07654553,sanchez2023kreinrutman} (see also \cite{MR2857021} and the references therein).

\smallskip

We consider a  Banach lattice $X$, which means that $X$ is a Banach space endowed with a  closed positive cone  $X_+$ (we write $f \ge 0$ if $f \in X_+$ and we recall that $f = f_+ - f_-$ with $f_\pm \in X_+$ for any $f \in X$. We also denote $|f| := f_+ + f_-$). We assume that $X$  is in duality with another Banach lattice $Y$, with  closed positive cone  $Y_+$,  so that the bracket $\langle \phi,f \rangle$ is well defined for any $f \in X$, $\phi \in Y$, and 
 that $f \in X_+$ (resp. $\phi \ge 0$) iff $\langle \psi , f\rangle \ge 0$ for any $\psi \in Y_+$ (resp. iff  $\langle \phi, g \rangle \ge 0$ for any $g \in X_+$), typically $X = Y'$ or $Y = X'$. 
  We write  
 $\psi \in Y_{++}$ if $\psi \in Y$ satisfies $\langle \psi, f \rangle > 0$ for any $f \in X_+ \backslash \{ 0 \}$. 
 
 \smallskip
 We consider a  positive and  conservative (or stochastic) semigroup $S = (S_t) = (S(t))$ on $X$, that means that  $S_t$ is a bounded linear mapping on $X$ such that 
 
 \begin{itemize}
 
 \item $S_t : X_+ \to X_+$ for any $t \ge 0$, 
 
  \item there exist  $\phi_1 \in Y_{++}$, $\| \phi_1 \| = 1$,   and a dual semigroup $S^*= S^*_t = S^*(t)$ on $Y$ such that  $S_t^* \phi_1 =  \phi_1$ for any $t \ge 0$. 
  More precisely, we assume that $S^*_t$ is a bounded linear mapping on $Y$ such that $\langle S(t) f, \phi \rangle = \langle  f, S^*(t)\phi \rangle$, for any $f \in X$, $\phi \in Y$ and $t \ge 0$, 
  and thus in particular $S^*_t : Y_+ \to Y_+$ for any $t \ge 0$. 
 \end{itemize} 
 
\smallskip

We denote by $\LL$ the generator of $S$ with domain $D(\LL)$. 
 For $\psi \in Y_+$, we define the seminorm
 $$
 [f]_\psi := \langle |f|, \psi \rangle, \ \forall \, f \in X.
  $$

\smallskip
In order to obtain a very accurate and constructive description of the longtime asymptotic behaviour of the semigroup $S$, we introduce additional assumptions.

\smallskip
$\bullet$  We first make the strong dissipativity assumption 
 \bear
\label{eq:NEWHarris-Primal-LyapunovCond}
\|   S (t) f  \|
&\le& C_0 e^{\lambda t}   \|  f  \| +   C_1 \int_0^t e^{\lambda(t-s)} [  S(s) f ]_{\phi_1}  ds,  
 \eear
for any $f \in X$ and $t \ge 0$, where $\lambda < 0$ and $C_i \in (0,\infty)$.
 
\smallskip
$\bullet$ Next, we make the slightly relaxed  Doblin-Harris  positivity  assumption 
   \bear
\label{eq:NEWDoblinHarris-primal}
&&  
S_T f \ge  \eta_{\eps,T} g_\eps [S_{T_0} f]_{\psi_\eps}, \quad \forall \, f \in X_+,  
 \eear
 for any  $T \ge T_1 >  T_0 \ge0$ and $\eps > 0$, where $ \eta_{\eps,T} > 0$, $g \in X_+ \backslash \{ 0 \}$ and $(\psi_\eps)$ is a  bounded and decreasing family  of $Y_{+} \backslash \{ 0 \}$. 
We say that the above condition is relaxed because we possibly have $T_0 > 0$ while the condition \eqref{eq:NEWDoblinHarris-primal}   holds with $T_0 = 0$ in the usual  Doblin-Harris.

\smallskip
$\bullet$ We finally assume the following  compatibility interpolation like condition  
  \bear
\label{eq:NEWHarris-LyapunovCondNpsieps}
&&  
[f]_{\phi_1} \le \xi_\eps  \| f \| + \Xi_\eps [f]_{\psi_\eps}, \ \forall \, f \in X,  \  \eps \in (0,1], 
 \eear
for two positive real numbers families $(\xi_\eps)$ and $(\Xi_\eps)$ such that $\xi_\eps\searrow0$ as $\eps \searrow 0$.

 \smallskip
 We refer to \cite{CGMM**} for a detailed discussion about these assumptions.

\begin{theo}\label{theo:KRDoblinHarris}
Consider a semigroup $S$ on a Banach lattice $X$ which satisfies the above conditions. 
Then, there exists a unique normalized positive stationary state  $f_1 \in D(\LL)$, that is  
$$
\LL f_1 = 0, \quad f_1 \ge 0, \quad   \langle \phi_1, f_1 \rangle = 1.
$$
 Furthermore, there exist some constructive constants $C \ge 1$ and $\lambda_2 < 0$ such that 
\beqn\label{eq:KRTh-constructiveRate}
\|S(t) f - \langle f,\phi_1 \rangle f_1   \| \le C e^{\lambda_2 t} 
\|  f - \langle f,\phi_1 \rangle f_1 \|
\eeqn
for any $f \in X$ and $t \ge 0$.
 \end{theo}

\begin{proof}[Sketch of the proof of Theorem~\ref{theo:KRDoblinHarris}]
We just allude the proof which is very similar to the proof of \cite[Theorem~7.1]{CGMM**} 

\medskip\noindent
{\sl Step 1. The Lyapunov condition.}   On the one hand, we classically have 
\beqn\label{eq:Stochastic}
[ S (t) f  ]_{\phi_1} \le [f  ]_{\phi_1}, \quad \forall \, t \ge 0, \ \forall \, f \in X. 
\eeqn
For $f \in X$, we may indeed write $f = f_+ - f_-$, $f_\pm \in X_+$, and then compute   
  \bean 
| S_t f | &\le& |S_t f_+| +  |  S_tf_- |
  \\
  &=&  S_t f_+  +  S_tf_-  = S_t |f|, 
  \eean
  where we have used  the positivity property of $S_t$ in the second line. We deduce 
    \begin{equation*}
   [S_t f ]_{\phi_1} \le  \langle S_t |f|, \phi_1 \rangle = \langle  |f|, S^*_t \phi_1 \rangle
  \end{equation*}
  and thus \eqref{eq:Stochastic}, because of  the stationarity property of $\phi_1$.

\smallskip
On the other hand, from \eqref{eq:NEWHarris-Primal-LyapunovCond} and \eqref{eq:Stochastic}, we have 
\bean
\|    S_t f  \|
\le  C_0 e^{\lambda t}   \|  f  \|  +   C_1 \int_0^t e^{\lambda (t-s)} [   f  ]_{\phi_1} ds , 
\eean
and we may thus choose $T \ge T_1$ large enough in such a way that 
\beqn\label{eq:LyapunovCondBIS}
\|    S_T  f  \| \le  \gamma_L   \|  f  \| +   K [ f  ]_{\phi_1}, 
\eeqn
with 
$$
\gamma_L := C_0 e^{\lambda T} \in (0,1), \quad
   K := C_1/\lambda.
$$

 \medskip
 \noindent
{\sl  Step~2. The Doblin-Harris condition.} Take $f \ge 0$ such that $\| f \| \le A [f]_{\phi_1}$ with $A > K/(1-\gamma_L)$.
We have 
\bean
\| S_{T_0} f\| &\le& (\gamma_L+K) \| f \|
\\
&\le& (\gamma_L+K) A [f]_{\phi_1}
\\
&=& (\gamma_L+K) A  [S_{T_0}f]_{\phi_1}
\\
&\le& (\gamma_L+K) A   (\xi_\eps \| S_{T_0}f \| + \Xi_\eps [ S_{T_0 }f ]_{\psi_\eps})  
 \eean\Black
for any $\eps > 0$, where we have used successively the growth estimate \eqref{eq:LyapunovCondBIS} in the first line, the condition on $f$ in the second line, 
the stationarity property of $\phi_1$ in the third line and the interpolation inequality \eqref{eq:NEWHarris-LyapunovCondNpsieps} in the last line. Choosing $\eps > 0$ small enough, we immediately obtain 
\bean
\| S_{T_0}f \| \le 2 (\gamma_L+K) A   \Xi_\eps [ S_{T_0 }f ]_{\psi_\eps}. 
 \eean
Together with 
$$
  [f]_{\phi_1} =   [S_{T_0}f]_{\phi_1}  \le 
   \| S_{T_0}f \|
$$
and the relaxed  Doblin-Harris  positivity condition \eqref{eq:NEWDoblinHarris-primal},  we conclude to the conditional Doblin-Harris  positivity estimate    
$$ 
 S_T f \ge c g_\eps [  f]_{\phi_1}
$$ 
for all $T\geq T_1$, with  $c^{-1} = c^{-1}_A := 2 (\gamma_L+K) A   \Xi_\eps \eta^{-1}_{\eps,T}  $. 
We may now classically improve the non-expensive estimate \eqref{eq:Stochastic} on the set $\NN := \{ f \in X; \langle \phi_1, f \rangle = 0 \}$. 
Take indeed $f \in \NN$ such that $\| f \| \le A [f]_{\phi_1}$. Observing that $[f_\pm]_{\phi_1} = [f]_{\phi_1}/2$ and thus  $\| f _\pm \| \le \| f \| \le 2A[f_\pm]_{\phi_1}$, the previous estimate tells us that 
$$
S_T f_\pm \ge \varrho g_\eps  \quad \varrho := c_{2A}  [  f]_{\phi_1}.
$$
Slightly modifying the arguments of Step~1, we compute now  
  \bean 
| S_t f | &\le& |S_t f_+ - \varrho g_\eps | +  |  S_tf_-  - \varrho g_\eps  |
  \\
  &=&   S_t |f| -  2 \varrho g_\eps. 
  \eean
  We deduce 
$$
   [S_t f ]_{\phi_1} \le    \langle  |f|,   \phi_1 \rangle - 2  \varrho \, \langle \phi_1, g_\eps \rangle  , 
 $$
 and thus   conclude to the {\it conditional coupling estimate}
 \beqn\label{eq:DoblinHarrisConditional}
   [S_t f ]_{\phi_1} \le    \gamma_H    [ f ]_{\phi_1}, 
\eeqn
  with $\gamma_H := 1 - 2 c_{2A} \langle \phi_1, g_\eps \rangle \in (0,1)$.  
 
\medskip
 \noindent
{\sl  Step~3.} We introduce a new equivalent norm $\Nt \cdot \Nt$ on $X$ defined by
 \begin{equation}\label{eq:newNormV}
\Nt f  \Nt    := [ f ]_{\phi_1}+ \beta \| f \|.
  \end{equation}
   
  Using the three properties \eqref{eq:Stochastic}, \eqref{eq:LyapunovCondBIS} and \eqref{eq:DoblinHarrisConditional}, 
we may prove that there exist $\beta > 0$ small enough and $\gamma \in (0,1)$ such that 
  \begin{equation}
    \label{eq:Harris-contrac}
  \Nt   S_T f   \Nt   \leq  \gamma     \Nt  f   \Nt ,   \quad
    \text{for any } f\in \NN. 
  \end{equation}
We refer to \cite[Proof of Theorem~7.1]{CGMM**} where this claim is established, and to \cite{MR2857021,zbMATH07654553} for previous variants. 
We then classically conclude, see again for instance \cite[Proof of Theorem~7.1]{CGMM**}.
\end{proof}

 %%%%%%%%%%%%%%%%%%%%%%%%%%%%%%%%%%%%%%%%%%%%%%%%%%%%%%%%%%%%%%%%%%%%

\section{Convergence in the large time asymptotic}  
\label{sec:asympt}
 
 %%%%%%%%%%%%%%%%%%%%%%%%%%%%%%%%%%%%%%%%%%%%%%%%%%%%%%%%%%%%%%%%%%%%

 In this section, we prove Theorem~\ref{theo-StabLinear} as a consequence of the previous analysis of the linear VCk equation, in particular Theorem~\ref{theo-Ultra}, and of the constructive Doblin-Harris stability result 
 stated in Theorem~\ref{theo:KRDoblinHarris}.

 \smallskip
 \begin{proof}[Proof of Theorem~\ref{theo-StabLinear}] We split the proof into fours steps, checking first that the conditions of Theorem~\ref{theo:KRDoblinHarris} are met.

 \smallskip\noindent
 {\sl Step~1. The stochastic semigroup and the strong dissipativity condition.} We proceed similarly as  during the proof of Theorem~\ref{theo-Exists}. We fix an admissible weight function $\omega$ and we consider the Banach lattice $X := L^2_\omega$.
We fix a  constant $a^* > 2 (a_*+y_*)$ and for any two constants $N,\cc \ge 0$ such that $(\cc + \cc^2)N \le a^*/2$, we set $\fraka := a_* + \cc^2 N \in [a_*,a^*]$,  $\frakb := y_* + \cc N \in [y_*,a^*]$, $K := \frakb-y$
and $\LL := \LLL_{\fraka,\frakb}$. Because of Theorem~\ref{eq:theo:existL2-1}, we know that $\LL$ generates a semigroup on $L^2_\omega$ which is associated to the linear equation \eqref{eq:VCklinear} with boundary conditions \eqref{eq:LinearVCkBd} and which is mass and  positivity conservative, so that $S_\LL$ is stochastic and $\phi_1 = 1$. 
  Introducing next the same splitting \eqref{eq:splittingL=A+B} and using \eqref{eq:DecaySBN}, the $L^2$ variant of \eqref{eq:UltraSBN}, \eqref{eq:DuhamelIterate}, the $L^2$ variant of  \eqref{eq:V&Wbd}, we deduce that 
$$
\| S_\LL(t) f_0 \|_{L^2_\omega} \le  C_V e^{-t} \| f_0 \|_{L^2_\omega} + \int_0^t C_W e^{-(t-s)} \| S_\LL (s) f_0 \|_{L^1} ds
$$
which is nothing but the strong dissipativity condition \eqref{eq:NEWHarris-Primal-LyapunovCond}.

\smallskip\noindent
 {\sl Step 2. The Doblin-Harris condition and the Doblin-Harris Theorem.} 
 For $0 \le f_0 \in L^2_\omega$, let us denote $f _t:= S_\LL(t) f_0$.  
From the Harnack inequality \eqref{eq:Harnack}, for any $T  > T_0 > 0$ and $\eps >0$, there exists  
 a constant $C \in (0,\infty)$, independent of $f_0$, so that 
\bean
f_T 
&\ge&  \inf_{\OO_\eps} f_{T}  {\bf 1}_{\OO_\eps}
\\
&\ge&  {1\over C} {1 \over |\OO_\eps|}  {\bf 1}_{\OO_\eps}
 \int_{\OO_\eps} f_{T_0}, 
\eean
what is nothing but \eqref{eq:NEWDoblinHarris-primal} with $g_\eps = \psi_\eps := {\bf 1}_{\OO_\eps}$.  
On the other hand, for any $f \in L^2_\omega$, we have 
\bean
\int_\OO |f| &=& \int_{\OO^c_\eps} |f| +  \int_{\OO_\eps} |f| 
\\
&\le& \Bigl( \int_{\OO^c_\eps} \omega^{-2} \Bigr)^{1/2} \| f \|_{L^2_\omega} + \int |f| {\bf 1}_{\OO_\eps}, 
\eean
so that \eqref{eq:NEWHarris-LyapunovCondNpsieps} holds true with 
$\Xi_\eps := 1$, $\xi_\eps := \| {\bf 1}_{\OO^c_\eps} \|_{L^2_\omega} \to 0$ because $\omega^{-1} \in L^2(\OO)$.   
As a consequence, of  Theorem~\ref{theo:KRDoblinHarris}, there exists a nonnegative and normalized steady state $\MMM_N$ which is furthermore asymptotically exponential stable.

 \smallskip\noindent
 {\sl Step~3. The fixed point argument.} 
  Take $0 \le F_* \in L^\infty_\omega$ such that $\| F_* \|_{L^1} = 1$. Choosing $\cc > 0$ small enough in such a way that $(\cc +\cc^2) \| F_* \|_{L^\infty_\omega} \le \eta^*$, we see that 
  $$
\frak F_* := \{ f \in L^\infty_\omega; \, f \ge 0, \, \| f \|_{L^1} = 1, \, \| f \|_{L^\infty_\omega} \le C_2 \| F_* \|_{L^\infty_\omega} \}
$$
is not empty because $F_* \in \frak F_*$.  From Steps 1 \& 2, we may apply Theorem~\ref{theo:KRDoblinHarris} and we deduce that  $S_\LL(t) F_* \to \MMM_N$ as $t \to \infty$. 
On the other hand, from \eqref{eq:SLbdLinfty}, 
we have $S_\LL(t) F_* \in \frak F_*$ for any $t \ge 0$.
We deduce that $\MMM_N \in \frak F_*$, and in particular 
\bean
0 \le \NN(\MMM_{N}) \le C_3 \| \gamma \MMM_N \|_{L^\infty_\omega} \le C_3 \|   \MMM_N \|_{L^\infty_\omega}  \le C_3 C_2 \| F_* \|_{L^\infty_\omega} \le {C_3 C_3 \over \cc + \cc^2} \eta^* = N^*. 
\eean
 We have thus built a mapping 
$$
\Lambda_* : [0,N^*] \to  [0,N^*] , \quad N \mapsto \Lambda_*(N) :=  \NN(\MMM_N)
$$
From Step 3 in the proof of  Theorem~\ref{theo-Exists}, the mapping $\Lambda$ is continuous. As a consequence, there exists at least a fix point $N^\sharp \in [0,N^*]$, so that $\Lambda_*(N^\sharp) = N^\sharp$. 
The function $\MMM := \MMM_{N^\sharp}$ is then a nonnegative normalized stationary solution to the nonlinear VCk equation \eqref{eq:VCk}-\eqref{eq:VCkBd} and it is linearly asymptotically exponential stable by construction.  
 \end{proof}

 \bigskip

\paragraph{\textbf{Acknowledgments.}}
The authors warmly thank Delphine Salort  
and  François Murat  for the enlightening discussions during the elaboration of this paper. 
This work is supported by the ANR project ChaMaNe (ANR-19-CE40-0024).

\bigskip
\bigskip
%%%%%%%%%%%%%%%
\bibliographystyle{plain}
%\bibliography{bib-VCneurons}

\end{document}